\documentclass{siamart251216}
\usepackage{algorithmic}

\usepackage{algorithm}
\usepackage{float}
\usepackage{amsfonts,amsmath,amssymb}
\usepackage{booktabs}
\usepackage{threeparttable}
\usepackage{graphicx}
\usepackage{multirow}
\usepackage{makecell}
\usepackage{subcaption}
\newcommand{\BA}{\mathbf{A}}
\newcommand{\BB}{\mathbf{B}}
\newcommand{\BC}{\mathbf{C}}
\newcommand{\BD}{\mathbf{D}}
\newcommand{\BE}{\mathbf{E}}
\newcommand{\BH}{\mathbf{H}}
\newcommand{\BI}{\mathbf{I}}
\newcommand{\BL}{\mathbf{L}}
\newcommand{\BM}{\mathbf{M}}
\newcommand{\BP}{\mathbf{P}}
\newcommand{\BQ}{\mathbf{Q}}
\newcommand{\BS}{\mathbf{S}}
\newcommand{\BT}{\mathbf{T}}
\newcommand{\BU}{\mathbf{U}}
\newcommand{\BV}{\mathbf{V}}
\newcommand{\BX}{\mathbf{X}}
\newcommand{\BZ}{\mathbf{Z}}
\newcommand{\BSigma}{\boldsymbol{\Sigma}}
\newcommand{\BOmega}{\boldsymbol{\Omega}}
\newcommand{\BPsi}{\boldsymbol{\Psi}}
\newcommand{\BPi}{\boldsymbol{\Pi}}
\newcommand{\bb}{\mathbf{b}}
\newcommand{\bc}{\mathbf{c}}

\newcommand{\be}{\mathbf{e}}
\newcommand{\bff}{\mathbf{f}}
\newcommand{\bG}{\mathbf{G}}
\newcommand{\bg}{\mathbf{g}}
\newcommand{\bn}{\mathbf{n}}
\newcommand{\bq}{\mathbf{q}}
\newcommand{\br}{\mathbf{r}}
\newcommand{\bs}{\mathbf{s}}
\newcommand{\bu}{\mathbf{u}}
\newcommand{\bv}{\mathbf{v}}
\newcommand{\bw}{\mathbf{w}}
\newcommand{\bx}{\mathbf{x}}
\newcommand{\by}{\mathbf{y}}
\newcommand{\bz}{\mathbf{z}}
\newcommand{\bphi}{\boldsymbol{\phi}}
\newcommand{\bpsi}{\boldsymbol{\psi}}
\newcommand{\bepsilon}{\boldsymbol{\epsilon}}
\newcommand{\half}{\frac{1}{2}}

\newcommand{\pzc}[1]{\textcolor{violet}{#1}}

\newtheorem{remark}{Remark}
\title{Low-rank SI-DSA and GMRES-DSA for Radiative Transfer Equation with Adaptive Accuracy Control}

\author{Chenxi Han\footnote{Department of Mathematics, The Hong Kong University of Science and Technology, Clear Water Bay, Kowloon, Hong Kong, China. E-mail: chenxi.han@connect.ust.hk.},\quad Lukas Einkemmer \footnote{ Department of Mathematics, University of Innsbruck, Technikerstrasse 13, Innsbruck, A-6020, Austria.
E-mail: lukas.einkemmer@uibk.ac.at.},\quad Wei Guo\footnote{Department of Mathematics and Statistics, Texas Tech University, Lubbock, TX, 70409, USA.},\quad   Zhichao Peng\footnote{Department of Mathematics, The Hong Kong University of Science and Technology, Clear Water Bay, Kowloon, Hong Kong, China. Corresponding author, E-mail: pengzhic@ust.hk.}}

\date{} 

\begin{document}
\maketitle

\begin{abstract}
Recently, efficient sweep-based low-rank iterative solvers have been
developed to reduce the computational costs of steady-state solves and implicit time integration of the radiative transfer equation (RTE). In particular, angular
sampling based on the discrete empirical interpolation method (DEIM) or on error-indicated randomized greedy selection enables efficient, non-intrusive reuse of existing transport sweep implementations.
However, randomized-greedy low-rank source iteration with diffusion synthetic acceleration (SI--DSA) enforces uniformly tight accuracy tolerances across all source iterations, thereby oversolving the early iterations, and it does not explicitly reuse angular bases from preceding iterations. More critically, like its full-rank counterpart, this approach
may lose effectiveness in multiscale problems with strong material
discontinuities.  We address these limitations through two complementary developments. First, we introduce adaptive accuracy control into randomized-greedy low-rank SI--DSA, reducing intermediate ranks and computational cost while maintaining the accuracy required for convergence. A two-stage hybrid sampling strategy can provide further acceleration by switching to DEIM-based sampling near convergence, once the preceding angular basis becomes sufficiently representative. Second, we develop an inexact flexible GMRES--DSA solver for the integral scalar-flux reformulation, in which sweep-based low-rank approximations are confined to the construction of the right-hand side and the matrix--vector products. The scalar-flux Krylov basis is not compressed, thereby preserving Arnoldi orthogonality in exact arithmetic and allowing established inexact-GMRES analysis to guide adaptive  accuracy control of low-rank operations.  We present numerical results on benchmark problems that demonstrate the effectiveness of the proposed methods.
\end{abstract}

%

\section{Introduction}

The radiative transfer equation (RTE) models the propagation and interaction of particles, including neutrons, photons, and protons, with a background medium. It arises in nuclear engineering, biomedical engineering, astrophysics, and atmospheric science. Its dependence on both position and angular direction makes the phase space high-dimensional and classical deterministic discretizations computationally costly. 

To mitigate this computational challenge, low-rank solvers leveraging low-rank matrix or tensor compressions in the phase space have been developed for RTE. Two principal approaches are dynamical low rank (DLR)
\cite{koch2007dynamical} and step-and-truncation (SAT)
\cite{kormann2015semi,dektor2021rank,guo2022low}. DLR evolves the solution on a low-rank manifold through tangent-space projection, whereas SAT advances it in an enlarged space and then truncates. DLR methods
for the RTE have been developed using first- and second-order formulations
\cite{peng2020low,ceruti2022rank,hauck2023predictor,haut2026efficient},
micro--macro decompositions
\cite{ding2021dynamical,einkemmer2024asymptotic,li2026asymptotic},
high-order-low-order method \cite{peng2021high},  pseudo-time
marching for proton transport
\cite{kusch2021robust,kusch2023robust}, domain decomposition \cite{brunner2026domain} and semi-Lagrangian method \cite{li2026asymptotic}. SAT methods have been
developed for single-group  \cite{sands2024high} and multigroup RTE \cite{truong2024tensor,deshpande2026multigroup,ortega2026tensor}. Rather than providing an exhaustive literature review, we refer readers to
\cite{bachmayr2023low,einkemmer2025review} for overviews of low-rank methods for
kinetic equations and general high-dimensional problems and to
\cite{ghahremani2024cross,dektor2025collocation,dektor2025interpolatory,sands2025adaptive,carrel2025interpolatory} for recent sampling-based developments.

Recently, low-rank source iteration (SI) or Richardson iterations have been developed for both steady-state problems \cite{bachmayr2026low,guo2026inexact,Guo2026HighlyER} and implicit time marching \cite{peng2023sweep,haut2026efficient}. 
Despite reducing storage, generic low-rank solvers based on inner low-rank Krylov methods \cite{kressner2011low} or alternating optimization \cite{dolgov2014alternating} may fail to achieve speedups over established fast full-rank iterative solvers on challenging problems \cite{guo2026inexact,ortega2026tensor}, and may even encounter convergence difficulties \cite{ortega2026tensor}. Reusing established transport sweeps  \cite{Adams2002FastIM,larsen2009advances} and diffusion synthetic acceleration (DSA)  \cite{Adams2002FastIM,larsen2009advances} preconditioners is therefore highly desirable for efficient low-rank iterative solvers. Transport sweeps efficiently solve the streaming--collision equations arising from discrete-ordinates angular discretization, while DSA accelerates convergence. However, the methods in \cite{bachmayr2026low,guo2026inexact} and the $P_N$-like method in \cite{haut2026efficient} employ second-order reformulations, which do not directly admit transport sweeps and may lead to ill-conditioned or singular systems in near-void or void regions. The DLR method in \cite{peng2023sweep} retains the first-order formulation and transport sweeps but requires intrusive modifications to the sweep implementation. In contrast, the low-rank SI--DSA methods for $S_N$ formulations in \cite{haut2026efficient,Guo2026HighlyER} non-intrusively reuse existing sweep implementations through DEIM-based and randomized greedy angular sampling, respectively. The latter method \cite{Guo2026HighlyER} also enables efficient rank adaptivity in the moderately high-rank regime.

Despite recent advances in sweep-based low-rank SI-DSA \cite{haut2026efficient,Guo2026HighlyER}, limitations persist. \textit{First}, the cost of individual low-rank updates in SI--DSA can be further reduced. Although \cite{Guo2026HighlyER} efficiently adapts the rank, it enforces uniformly tight tolerances, unnecessarily solving early updates to high accuracy, and does not explicitly reuse angular bases from previous iterations.
\textit{Second}, SI--DSA can lose its effectiveness for multiscale problems with strong material discontinuities \cite{warsa2004krylov}. Low-rank approximation reduces the cost per iteration but does not overcome this intrinsic convergence limitation, motivating its integration with Krylov solvers.
In this paper, we address these limitations through the following two methodological developments.
\begin{enumerate}
\item 
\textbf{Adaptive accuracy control and hybrid angular sampling for low-rank SI--DSA.}
We propose to adaptively control the inner low-rank accuracy of sweep-based SI--DSA with randomized greedy angular sampling \cite{Guo2026HighlyER}, thereby avoiding unnecessarily high-rank computations during early outer iterations. To reduce computational cost further, we introduce a two-stage hybrid sampling strategy. Randomized greedy sampling adapts the rank without a prior angular basis during early iterations; near convergence, we switch to a steady-state variant of the DEIM-based DLR method \cite{haut2026efficient} to reuse the angular basis from the preceding iteration once it becomes sufficiently representative.
    
    \item \textbf{Inexact generalized minimal residual method (GMRES)-DSA for scalar flux with low-rank matrix--vector products.} To address efficiency loss of SI-DSA for multiscale problems with strong discontinuity, we
    develop an inexact GMRES--DSA solving the integral  reformulation for scalar-flux, using sweep-based low-rank approximations of the right-hand side and matrix--vector products. Crucially, our method confines low-rank compression to the right-hand side
    and matrix--vector products. Unlike classical low-rank Krylov methods
    \cite{kressner2011low,dolgov2013tt}, it does not truncate the Krylov basis and therefore retains standard Arnoldi orthogonality in exact arithmetic. This allows the residual-gap analysis for inexact GMRES \cite{simoncini2003theory} relying on this orthogonality to be extended to flexible GMRES with an inexact right-hand side, which in turn guides the adaptive accuracy control of the low-rank approximations.
\end{enumerate}

The remainder of the paper is organized as follows. Sec.~\ref{sec:background} reviews the discretization, the classical SI--DSA and GMRES--DSA solvers, and motivates the proposed developments. Sec.~\ref{sec:sweep-lr-si-dsa} reviews the core ideas of sweep-based low-rank SI-DSA in \cite{haut2026efficient,Guo2026HighlyER} and extends the DEIM-based approach in \cite{haut2026efficient} from implicit time integration to steady-state iterations. Sec.~\ref{sec:hybrid-lr-si} introduces adaptive accuracy control and a two-stage hybrid sampling strategy for SI-DSA. Sec.~\ref{sec:ta-lr-fgmres} integrates the sweep-based low-rank solver into GMRES--DSA and provides theoretical guidelines for adaptively controlling the accuracy of low-rank operations. Sec.~\ref{sec:numerical} reports numerical results and Sec.~\ref{sec:conclusion} draws conclusions.

\section{Background and motivation}\label{sec:background}
In this paper, we consider the steady-state, single energy group, linear RTE with isotropic scattering and source:
\begin{subequations}\label{eq:rte-model}
\begin{align}
&\BOmega\cdot\nabla_{\bx}
\psi(\bx,\BOmega)
+
\sigma_t(\bx)
\psi(\bx,\BOmega)
=
\sigma_s(\bx)\phi(\bx)
+
G(\bx),
(\bx,\BOmega)
\in\Gamma_{\bx}\times\mathbb{S}^2,
\label{eq:rte-model-a}
\\
&\phi(\bx)
=
\frac{1}{4\pi}
\int_{\mathbb{S}^2}
\psi(\bx,\BOmega)
\,\mathrm{d}\BOmega,\quad
\bx\in\Gamma_{\bx},
\label{eq:rte-model-b}
\\
&\psi(\bx,\BOmega)
=
g(\bx),\quad
\bx\in\partial\Gamma_{\bx},\;\BOmega\cdot\bn(\bx)<0.
\label{eq:inflow-bc}
\end{align}
\end{subequations}
Here, $\psi(\bx,\BOmega)$ is the angular flux at spatial location $\bx\in\Gamma_{\bx}$ and angular direction $\BOmega=(\Omega_{x_1},\Omega_{x_2},\Omega_{x_3})\in\mathbb{S}^2$, while $\phi(\bx)$ is the scalar flux. The scattering, absorption, and total cross sections are denoted by $\sigma_s(\bx)\geq0$, $\sigma_a(\bx)\geq0$, and $\sigma_t(\bx)=\sigma_s(\bx)+\sigma_a(\bx)$. An isotropic inflow boundary condition is posed in \eqref{eq:inflow-bc}, with $\bn(\bx)$ denoting the outward normal at $\bx\in\partial\Gamma_{\bx}$.
When the scattering effect goes to infinity, RTE formally converges to its diffusion limit: as $\sigma_s\rightarrow\infty$, $\psi(\bx,\BOmega)
\to\phi(\bx)$, and $\phi(\bx)$ satisfies
\begin{equation}\label{eq:formal-diffusion-limit}
    -\nabla\cdot
    \Big(
        \frac{1}{3\sigma_t(\bx)}
        \nabla\phi(\bx)
    \Big)
    +
    \sigma_a\phi(\bx)
    =
    G(\bx).
\end{equation}
This section briefly reviews numerical discretization and full-rank and low-rank iterative solvers for steady-state RTE, and motivates further enhancement of low-rank solvers through adaptive accuracy control and Krylov method.

\subsection{Discretization}
\label{sec:discretization}
We apply discrete ordinates ($S_N$) angular discretization \cite{pomraning2005equations}  and the upwind discontinuous Galerkin (DG) spatial discretization \cite{adams2001discontinuous}, which is proved to be asymptotic preserving \cite{guermond2010asymptotic} with linear or higher order polynomials.

\textbf{$S_N$ angular discretization.} We solve RTE at  quadrature points $\{\BOmega_j\}_{j=1}^{N_\Omega}$ associated with normalized quadrature weights $\{\omega_j\}_{j=1}^{N_\Omega}$ satisfying $\sum_{j=1}^{N_\Omega}\omega_j=1$: 
\begin{subequations}
\label{eq:sn-system}
\begin{align}
&\BOmega_j\cdot\nabla\psi_j
    +\sigma_t\psi_j
    =
    \sigma_s\phi+G,
    \qquad
    1\leq j\leq N_\Omega,
    \label{eq:sn-system-a}
    \\
    &\phi
    =
\sum_{j=1}^{N_\Omega}\omega_j\psi_j,
\quad
    \psi_j(\bx)|_{\bx\in\partial\Gamma_{\bx}}
    =
    g_j,\;\text{if}\;\; \BOmega_j\cdot\bn(\bx)<0.
    \label{eq:sn-system-c}
\end{align}
\end{subequations}
Here, $\psi(\bx,\BOmega_j)\approx\psi_j(\bx)$. In this paper, we use Gauss-Chebyshev quadrature rule detailed in Appx. \ref{app:quadrature}.

\textbf{Discrete $S_N$ system.} The details of the upwind DG spatial discretization are omitted here; the interested reader is referred to \cite{adams2001discontinuous,Guo2026HighlyER}.  Applying this scheme to the $S_N$ equations \eqref{eq:sn-system} yields the following discrete system:
\begin{equation}
\label{eq:fully-discrete-system}
    \left(
        \BD_j+\BSigma_t
    \right)
    \bpsi_j
    =
    \BSigma_s\bphi
    +
    \bG
    +
    \bg^{\mathrm{bc}}_j,
    \;
    1\leq j\leq N_\Omega,\quad\text{where}\;     \bphi
    =
    \sum_{j=1}^{N_\Omega}
    \omega_j\bpsi_j.
\end{equation}
The vectors $\bpsi_j\in\mathbb{R}^{N_{\bx}}$ and $\bphi\in\mathbb{R}^{N_{\bx}}$ collect spatial degrees of freedom (DOFs) for the angular flux $\psi_j(\bx)$ and scalar flux $\phi(\bx)$. The matrices $\BD_j,\BSigma_t,\BSigma_s\in\mathbb{R}^{N_{\bx}\times N_{\bx}}$ discretize $\BOmega_j\cdot\nabla$, the total cross section, and the scattering cross section, respectively. The vectors $\bG\in\mathbb{R}^{N_{\bx}}$ and $\bg^{\mathrm{bc}}_j\in\mathbb{R}^{N_{\bx}}$ represent the source and inflow-boundary contributions. Throughout the paper we collect the angular quadrature weights in the vector
$\bw=(\omega_1,\ldots,\omega_{N_\Omega})^T$ and in the diagonal matrix
$\BD_\Omega=\operatorname{diag}(\bw)$. Writing
$\BPsi=[\,\bpsi_1,\ldots,\bpsi_{N_\Omega}]\in\mathbb{R}^{N_{\bx}\times N_\Omega}$
for the angular-flux matrix, \eqref{eq:fully-discrete-system} leads to
$\bphi=\BPsi\bw$.

For simplicity, we focus on 2D $X$-$Y$ geometry and Cartesian grid. In this case, the discrete transport operator $\BD_j$ can be simplified as
\begin{equation}
    \BD_j=\sum_{l=1}^{d} \Omega_{j,x_l}\, \BD_{x_l}^{\operatorname{sign}(\Omega_{j,x_l})},
\end{equation}
where $d=2$ is the spatial dimension and $\BD_{x_l}^{\pm}$ are the left- and right-going upwind discretizations of $\partial_{x_l}$. Details of $\BD_j$, $\BSigma_t$, $\BSigma_s$ can be found in the Appx. B of \cite{Guo2026HighlyER}.

Tab.~\ref{tab:notation} in Appx. \ref{sec:notation} collects the notation used throughout the paper. We
write Latin vectors and matrices in upright bold ($\bx$, $\BA$) and their Greek
counterparts in bold italic ($\bphi$, $\BSigma_s$). Discrete quantities are
distinguished from their continuous counterparts by this bold face, e.g.
$\phi(\bx)$ versus $\bphi$.

\subsection{Source iteration, diffusion synthetic acceleration, transport sweeps and Krylov extension}
\label{sec:si-dsa}
Following \cite{Adams2002FastIM,larsen2009advances}, we briefly review source iteration with diffusion synthetic acceleration (SI-DSA) and transport sweeps to solve the discrete $S_N$ system \eqref{eq:fully-discrete-system}.
Each iteration of SI-DSA consists of the following two steps. 
\begin{enumerate}
    \item \textbf{Source iteration (SI).}
    Given the scalar flux from the previous iteration (or an initial guess), we update the
    angular flux by solving the SI equation 
    \begin{equation}
    \label{eq:si-angular-update}
        \left(
            \BD_j+\BSigma_t
        \right)
        \bpsi^{(n)}_j
        =
        \BSigma_s
        \bphi^{(n-1)}
        +
        \bG
        +
        \bg^{\mathrm{bc}}_j,\;
        j=1,\ldots,N_\Omega,
    \end{equation}
    for each angle, and update the scalar flux as $\bphi^{(n-\half)}=\sum_{j=1}^{N_\Omega}\omega_j\bpsi_j^{(n)}.$
    
    \item \textbf{Diffusion synthetic acceleration (DSA).} To accelerate convergence, we correct the scalar flux as
$\bphi^{(n)}=\bphi^{(n-\half)}+\delta\bphi^{(n)}$, where
 $\delta\bphi^{(n)}$ is obtained by solving a discretized diffusion limit of the error equation:
\begin{equation}
\label{eq:dsa-discrete-correction}
\BA_{\mathrm{diff}}\delta\bphi^{(n)}
=
\BSigma_s
\left(
    \bphi^{(n-\half)}
    -
    \bphi^{(n-1)}
\right),
\end{equation}
where $\BA_{\mathrm{diff}}\in\mathbb{R}^{N_{\bx}\times N_{\bx}}$ arises from a consistent discretization of 
$-\nabla\cdot\big(\frac{1}{3\sigma_t}\nabla\big)
+\sigma_a$ \cite{adams1992diffusion,Adams2002FastIM}.
We note that DSA serves as a crucial preconditioning step, and without it, SI can converge arbitrarily slow \cite{Adams2002FastIM,larsen2009advances}.
\end{enumerate}

\textbf{Transport sweeps.} The streaming-plus-collision equation in the SI step \eqref{eq:si-angular-update} is often solved by transport sweeps \cite{Adams2002FastIM,larsen2009advances}. With an upwind spatial discretization, the discrete operator $\BD_j+\BSigma_t$ becomes block lower triangular if the spatial elements are reordered along the upwind direction corresponding to $\BOmega_j$. Hence, equation \eqref{eq:si-angular-update} can be solved with a transport sweep, i.e. one iteration of matrix-free block Gauss-Seidel sweeping along the upwind direction. We refer to Sec. 2.2 
of \cite{Guo2026HighlyER} for more details.

\textbf{Sweep-based GMRES for integral scalar flux reformulation.}
When extending SI-DSA to GMRES in \cite{warsa2004krylov}, an integral scalar flux reformulation is applied. This reformulation is derived by inverting the discrete transport operator $\BD_j+\BSigma_t$ for all angles on both sides of \eqref{eq:fully-discrete-system} and then numerically integrating the angular flux:
\begin{align}
    &\bphi=\sum_{j=1}^{N_\Omega}\omega_j(\BD_j+\BSigma_t)^{-1}(\BSigma_s\bphi+\bG+\bg_j^{\mathrm{bc}})\\
    \Leftrightarrow&
    \left(\BI-\Big(\sum_{j=1}^{N_\Omega}\omega_j(\BD_j+\BSigma_t)^{-1}\Big)\BSigma_s\right)\bphi = \sum_{j=1}^{N_\Omega}\omega_j(\BD_j+\BSigma_t)^{-1}(\bG+\bg_j^{\mathrm{bc}}).
    \label{eq:memory-efficient-reformulation-long}
\end{align}
For simplicity, we define $\BT=\sum_{j=1}^{N_\Omega}\omega_j(\BD_j+\BSigma_t)^{-1}$ and $\bb=\sum_{j=1}^{N_\Omega}\omega_j(\BD_j+\BSigma_t)^{-1}(\bG+\bg_j^{\mathrm{bc}})$, then equation \eqref{eq:memory-efficient-reformulation-long} can be rewritten as 
\begin{equation}
   (\BI-\BT\BSigma_s)\bphi=\bb.
   \label{eq:memory-efficient-reformulation}
\end{equation}
In practice, the matrix-vector product determined by $\BT\in\mathbb{R}^{N_{\bx}\times N_{\bx}}$ is implemented by transport sweeps and numerical integration with costs of $O(N_{\bx}N_\Omega)$. 

 In \cite{warsa2004krylov}, SI-DSA is rewritten as a left-preconditioned Richardson iteration:
\begin{equation}
\bphi^{(n)}=\bphi^{(n-1)}+(\BI+\BA_{\mathrm{diff}}^{-1}\BSigma_s)\br^{(n-1)},
\label{eq:si-richardson}
\end{equation}
where the residual $\br^{(n-1)}=\bb-(\BI-\BT\BSigma_s)\bphi^{(n-1)}=\bphi^{(n-\half)}-\bphi^{(n-1)}$. When extending to the Krylov framework, \cite{warsa2004krylov}
directly uses GMRES to solve \eqref{eq:memory-efficient-reformulation} with the DSA preconditioner, $\BM_{\mathrm{DSA}}$,  defined through its action
\begin{equation} \BM_{\mathrm{DSA}}^{-1}:=\BI+\BA_{\mathrm{diff}}^{-1}\BSigma_s.
    \label{eq:dsa-preconditioner}
\end{equation}



\subsection{Low-rank iterative solvers for RTE\label{sec:lr-background}}
To mitigate the substantial storage and computational costs due to the high dimensionality of RTE, the angular flux can be approximated by a low-rank matrix factorization:
\begin{equation}
\BPsi   =\begin{bmatrix}
            \bpsi_1 & \dots & \bpsi_{N_\Omega}
            \end{bmatrix} 
\approx \BX\BS\BV^T,
\end{equation}
where $\BPsi\in\mathbb{R}^{N_{\bx}\times N_\Omega}$ is the matrix representation of the full-rank angular flux, $\BX\in\mathbb{R}^{N_{\bx}\times r}$ denotes the low-rank spatial basis, $\BV\in\mathbb{R}^{N_\Omega\times r}$ denotes the low-rank angular basis, and $\BS\in\mathbb{R}^{r\times r}$ denotes the low-rank core, and $r$ is the rank. Higher order tensor decompositions can also be applied with further decompositions in the physical or angular space. The low-rank solutions can be obtained by dynamical low-rank (DLR) \cite{koch2007dynamical} or step-and-truncation/rank-adaptive methods \cite{bachmayr2017iterative}.

For steady-state solves or implicit time marching of RTE, low-rank iterative solvers have been designed based on DLR  in \cite{peng2023sweep}, preconditioned low-rank Richardson iterations for the second-order reformulation of RTE \cite{bachmayr2026low,guo2026inexact}, spatial projection and angular sampling inside or outside SI \cite{Guo2026HighlyER,haut2026efficient}.
In particular, based on angular sampling through a greedy randomized approach inside SI \cite{Guo2026HighlyER} or discrete empirical interpolation method (DEIM) across time steps \cite{haut2026efficient}, \cite{haut2026efficient,Guo2026HighlyER} are able to reuse existing transport sweeps implementations non-intrusively.

\subsection{Motivations\label{sec:motivation}}
\begin{figure}[H]
    \centering
    \includegraphics[width=0.45\textwidth]{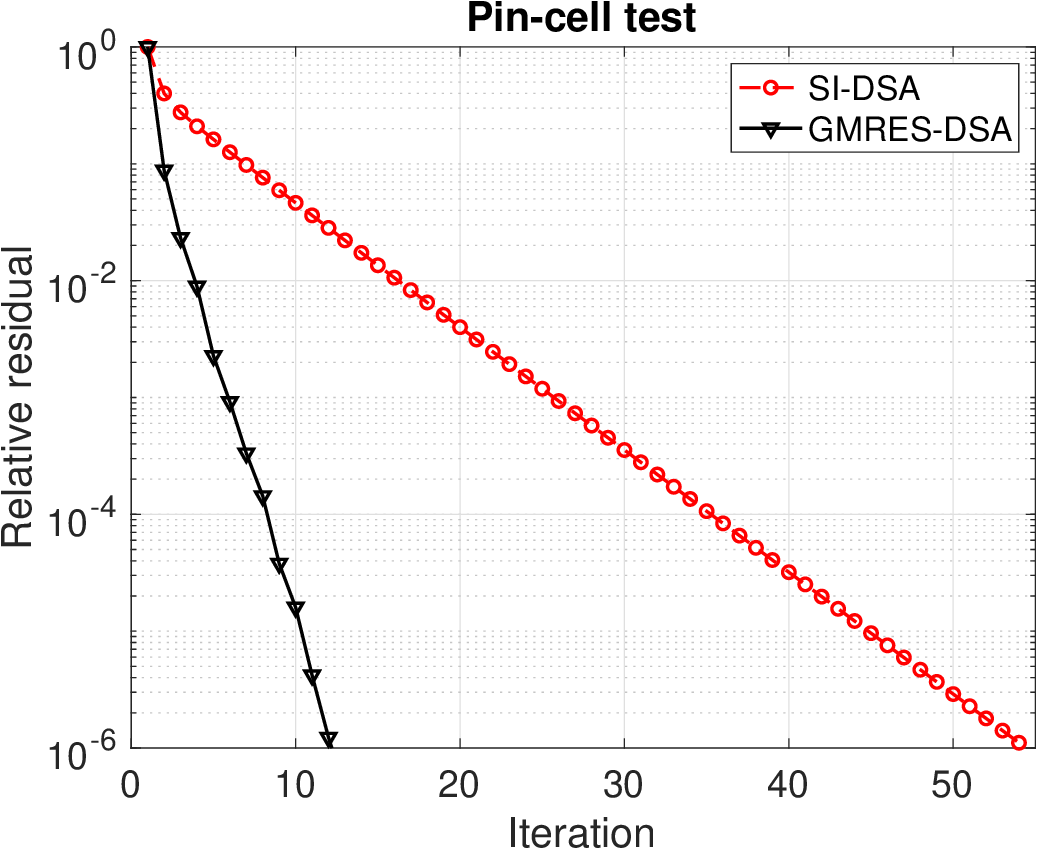}
    \caption{ Convergence history of the full-rank SI and GMRES. A partially consistent DSA preconditioner is applied following \cite{Adams2002FastIM}.}
    \label{fig:adaptive-accuracy}
\end{figure}
Despite recent developments of sweep-based low-rank SI-DSA  \cite{haut2026efficient,Guo2026HighlyER}, limitations remain.

\textit{First}, unlike low-rank iterative solvers with adaptive accuracy control \cite{bachmayr2017iterative,bachmayr2026low,appelo2025lraa}, the method in \cite{Guo2026HighlyER} enforces uniformly tight tolerances throughout source iterations. Consequently, early SI updates are unnecessarily solved to high accuracy while the iterates remain far from convergence. Moreover, though its randomized greedy angular sampling adapts the rank without a prior angular basis, it does not explicitly reuse bases from preceding iterations. Such reuse offers an additional opportunity for computational savings near convergence. The DEIM-based approach \cite{haut2026efficient} exploits the angular basis from the previous time step for angular sampling. In the steady-state setting, however, no such basis is available at initialization, and the basis from preceding iterations may not be sufficiently representative when the angular subspace is still changing substantially.

\textit{Second}, reducing per-iteration costs does not overcome the intrinsic convergence limitations of SI--DSA. Fourier analysis and numerical tests in \cite{warsa2004krylov} show that SI--DSA can lose effectiveness in multidimensional problems with highly diffusive regions and strong material discontinuities. Low-rank SI--DSA inherits this limitation. As presented in Fig. \ref{fig:adaptive-accuracy}, DSA-preconditioned GMRES for the integral scalar-flux reformulation \eqref{eq:memory-efficient-reformulation} can mitigate such effectiveness reduction \cite{warsa2004krylov}.

We address these limitations through two complementary developments.

\textbf{Adaptive accuracy control and hybrid angular sampling for low-rank SI--DSA.}
Guided by convergence analysis of inexact Richardson iteration, we adaptively control the accuracy of each low-rank SI solve, allowing cheaper, lower-rank updates during early outer iterations and tightening the tolerances as convergence proceeds. For further acceleration, we introduce a two-stage hybrid sampling strategy. Randomized greedy sampling builds the low-rank approximation during early iterations without requiring a prior angular basis. Near convergence, we switch to a steady-state variant of the DEIM-based method to reuse the preceding angular basis once it becomes sufficiently representative.

\textbf{GMRES-DSA solving the integral  reformulation for scalar flux through low-rank matrix-vector products.} To address the convergence limitations of SI--DSA, we develop an inexact GMRES--DSA solver for the integral scalar-flux reformulation. Specifically, we use sweep-based low-rank transport solves to construct the right-hand side and perform matrix--vector products. By confining low-rank approximations to these operations, we avoid compressing the scalar-flux Krylov basis and thus preserve standard Arnoldi orthogonality in exact arithmetic. Consequently, the residual-gap analysis for inexact GMRES \cite{simoncini2003theory}, which relies on this orthogonality, can be extended to our setting, guiding adaptive accuracy control of low-rank approximations. In contrast, directly applying classical low-rank Krylov methods to an angular-flux formulation \cite{kressner2011low,dolgov2013tt} introduces truncations during the construction of the Krylov basis, potentially compromising orthogonality unless carefully designed truncation and reorthogonalization procedures are employed \cite{palitta2021convergence}.

\section{Sweep-based low-rank SI-DSA}
\label{sec:sweep-lr-si-dsa}
In this section, we briefly review the sweep-based low-rank SI-DSA via angular sampling based on a randomized greedy approach \cite{Guo2026HighlyER} and DEIM \cite{haut2026efficient}. Though both approaches rely on angular sampling, they differ in how the angular directions are selected and consequently offer complementary strengths. Here, we first review the randomized greedy approach in Sec. \ref{sec:rss-review}, and then extends the DEIM-based approach from implicit time marching to subsequent steady-state iterations in Sec. \ref{sec:deim-lr-si}.

\subsection{Sweep-based rank-adaptive SI via randomized greedy angle sampling\label{sec:rss-review}}
To approximate the angular flux in low-rank form
$\BPsi\approx\BX\BS\BV^{T}$ with $\BX\in\mathbb{R}^{N_{\bx}\times r}$,
$\BS\in\mathbb{R}^{r\times r}$, $\BV\in\mathbb{R}^{N_\Omega\times r}$, \cite{Guo2026HighlyER} introduces an inner-loop iteration that selects representative angular directions to construct a low-rank spatial basis via a randomized greedy sampling inside each SI step. 

Each inner-loop iteration $i$ consists of three steps:
(i)~\emph{transport sweeps} for a small batch of $p$  sampled angles, using standard one-angle sweep implementations to compute
$\bpsi_{s_1^{(i)}},\ldots,\bpsi_{s_p^{(i)}}$;
(ii)~an \emph{incremental rank-$p$ augmentation} of the spatial basis for
$\operatorname{span}
[\,\BX^{(n,i-1)},\bpsi_{s_1^{(i)}},\ldots,\bpsi_{s_p^{(i)}}\,]
$
using two-pass modified Gram--Schmidt with adaptive reorthogonalization
\cite{daniel1976reorthogonalization}. This preserves the nested structure
$\BX^{(n,i)}=[\,\BX^{(n,i-1)},\BP^{(n,i)}\,]$ and enables incremental updates
of the Galerkin-projected operators; and
(iii)~\emph{residual-based greedy random sampling}, in which $q>p$
 unsampled candidate angles are randomly drawn and their reduced
approximations and residuals are computed by Galerkin projection of their SI equations \eqref{eq:si-angular-update} onto
$\operatorname{span}(\BX^{(n,i)})$.

The inner loop terminates if both the largest candidate residual and the
change in scalar flux between consecutive inner iterates fall below prescribed
thresholds. Otherwise, $p$ candidates with the largest residuals are
selected for the next augmentation. No intermediate truncation is performed
during the inner loop. 

This strategy balances sampling quality and cost, enabling efficient
rank adaptation without prior information. It constructs a reduced space only
slightly larger than the effective rank, thereby avoiding the
rank-proportional over-augmentation of standard SAT and BUG methods. The full
algorithm is summarized in Alg.~\ref{alg:lr-si}.

\begin{algorithm}[h]
\caption{Rank-adaptive low-rank SI (LR-SI) based on randomized greedy sampling.\label{alg:lr-si}}
\begin{algorithmic}[1]
\STATE{\textbf{Input}: Scalar flux $\bphi^{(n-1)}$ from the previous outer
iteration, tolerances $\varepsilon_{\mathrm{res}},\varepsilon_{\mathrm{diff}},
\varepsilon_{\mathrm{SVD}}$, batch size $p$, candidate-set size $q>p$, and
maximum number of inner iterations $i_{\max}$.}
\STATE{\textbf{Initialization}: Randomly sample $p$ initial angles
$\{\BOmega_{s_1^{(1)}},\dots,\BOmega_{s_p^{(1)}}\}$; set
$\BX^{(n,0)}:=[]$ and $\bphi^{(n)}_{\mathrm{old}}:=\bphi^{(n-1)}$.}
\FOR{$i=1,2,\dots,i_{\max}$}
    \STATE{\textbf{Transport sweeps}: Solve
    $(\BD_j+\BSigma_t)\bpsi_j^{(n)}=\BSigma_s\bphi^{(n-1)}+\bG+\bg_j^{\mathrm{bc}}$
    for $j=s_1^{(i)},\dots,s_p^{(i)}$.}
    \STATE{\textbf{Incremental basis update}: 
    $\BX^{(n,i)}=
    \operatorname{orth}[\BX^{(n,i-1)},\bpsi_{s_1^{(i)}},\dots,\bpsi_{s_p^{(i)}}]$
    via MGS2 with adaptive reorthogonalization; update the Galerkin-projected
    operators incrementally exploiting
    $\BX^{(n,i)}=[\,\BX^{(n,i-1)},\BP^{(n,i)}\,]$.}
    \STATE{\textbf{Greedy random subsampling}: Randomly draw $q$ previously
    unsampled angles to form $\mathcal{S}^{(i)}$; for each
    $\BOmega_j\in\mathcal{S}^{(i)}$, solve the reduced problem
    \begin{equation}
(\BX^{(n,i)})^T(\BD_j+\BSigma_t)\BX^{(n,i)}\bc_j^{(n,i)}
=
(\BX^{(n,i)})^T(\BSigma_s\bphi^{(n-1)}+\bG+\bg^{\mathrm{bc}}_j)
\label{eq:galerkin-projection}
    \end{equation}
     and evaluate the residual $\br_j^{(i)}$.}
    \IF{$\max_{\BOmega_j\in\mathcal{S}^{(i)}}\|\br_j^{(i)}\|_2
        <\varepsilon_{\mathrm{res}}$}
        \STATE{Solve \eqref{eq:galerkin-projection} for all remaining
        unsampled angles to form $\BC^{(n,i)}=[\bc_1^{(n,i)},\dots,\bc_{N_\Omega}^{(n,i)}]$, and compute
        $\bphi^{(n,i)}=\BX^{(n,i)}\bigl(\BC^{(n,i)}\bw\bigr)$.}
        \IF{$\|\bphi^{(n,i)}-\bphi^{(n)}_{\mathrm{old}}\|_2
            \le\varepsilon_{\mathrm{diff}}$}
            \STATE{Set $\bphi^{(n,*)}:=\bphi^{(n,i)}$; \textbf{break}.}
        \ELSE
            \STATE{Set $\bphi^{(n)}_{\mathrm{old}}:=\bphi^{(n,i)}$.}
        \ENDIF
    \ENDIF
    \STATE{\textbf{Angle selection}: Greedily select the $p$ angles in
    $\mathcal{S}^{(i)}$ with the largest residuals as
    $\{\BOmega_{s_1^{(i+1)}},\dots,\BOmega_{s_p^{(i+1)}}\}$.}
\ENDFOR
\STATE{\textbf{Optional truncation}: If the low-rank angular flux is
required, compute the truncated SVD
$\BC^{(n,i)}\approx\BU^{(n)}\BS^{(n)}(\BV^{(n)})^{T}$ with tolerance
$\varepsilon_{\mathrm{SVD}}$ and set
$\BX^{(n)}:=\BX^{(n,i)}\BU^{(n)}$.}
\STATE{\textbf{Output}: Scalar flux $\bphi^{(n,*)}$ and, optionally, the
rank-$r_n$ factorization $\BX^{(n)}\BS^{(n)}(\BV^{(n)})^{T}$.}
\end{algorithmic}
\end{algorithm}

\subsection{DEIM-based low-rank SI\label{sec:deim-lr-si}}

We next adapt the DEIM-based sweep-based SI  in
\cite{haut2026efficient} for implicit time marching to the steady transport equation  
by viewing successive outer-loop steady-state iterations as pseudo time-marching.

At the $n$-th iteration, we start from the low-rank factors $\BPsi^{(n)}
\approx
\BX^{(n)}\BS^{(n)}(\BV^{(n)})^T$ with 
$\BX^{(n)}\in\mathbb{R}^{N_{\bx}\times r}$, $\BS^{(n)}\in\mathbb{R}^{r\times r}$, 
$\BV^{(n)}\in\mathbb{R}^{N_\Omega\times r}.$
One iteration proceeds through the following four steps.

\textbf{(i) DEIM-based angle sampling.}
We first apply DEIM in 
 Appx.~\ref{app:deim-sampling-alg} to the  angular basis $\BV^{(n)}$ to select
$r$ representative angles, $\BOmega_{p_1},\ldots,\BOmega_{p_r}$, and define the index set for sampled angles as $\mathcal{I}=\{p_1,\dots,p_r\}$.

\textbf{(ii) Sweep-based spatial basis update.}
With $\bphi^{(n+1,0)}=\bphi^{(n)}$ as the initial guess, we perform an inner-loop SI-DSA solve
using transport sweeps only along the DEIM-selected directions. Using MATLAB indexing notation, the $\ell$-th inner SI iteration solves
\begin{subequations}
\begin{align}
&(\BD_{p_j}+\BSigma_t)
\bpsi_{p_j}^{(n+1,\ell)}
=
\BSigma_s\bphi^{(n+1,\ell-1)}
+\bG+\bg_{p_j}^{\mathrm{bc}},
\qquad j=1,\ldots,r,
\\
&\bphi^{(n+1,\ell-\half)}
=
\big[
\bpsi_{p_1}^{(n+1,\ell)},\ldots,\bpsi_{p_r}^{(n+1,\ell)}
\big]
\Big(\BV^{(n)}(\mathcal{I},:)\Big)^{-T}
(\BV^{(n)})^T\bw.
\end{align}
\label{eq:deim_column_sweep}
\end{subequations}
Here $\bw=(\omega_1,\ldots,\omega_{N_\Omega})^T$ collects angular
quadrature weights. The reconstructed scalar flux
$\bphi^{(n+1,\ell-\half)}$ is then used in the standard DSA correction.
Upon the inner-loop convergence, the updated spatial basis is obtained by
$\BX^{(n+1)}
=
\operatorname{orth}_x(\BPsi_{\mathrm{c}}^\star),$
where $\BPsi_{\mathrm{c}}^\star$ collects the sampled angular fluxes at the final inner iteration.
When 
\begin{equation}
\big\|
\bphi^{(n+1,\ell)}
-
\bphi^{(n+1,\ell-1)}
\big\|_{\infty}
/
\big\|
\bphi^{(n+1,\ell)}
\big\|_{\infty}
\leq \varepsilon_{\mathrm{DEIM-SI}}, \label{eq:inner-deim-stop}
\end{equation}
we stop the inner-loop iterations and set $\bphi^{(n+1)}=\bphi^{(n+1,\ell)}$.

\textbf{(iii) Projection-based angular basis update.}
With $\BX^{(n+1)}$ fixed, we project the steady transport equation onto the
updated spatial space. For each angle $\BOmega_j$, the reduced
angular coefficient $\BL_j\in\mathbb{R}^r$ satisfies
\begin{subequations}
\begin{align}
&\BB_j\BL_j
-
\BM_s
\overline{\BL}
=
(\BX^{(n+1)})^T
(\bG+\bg_j^{\mathrm{bc}})=\bff_j,
\\
\text{where}\;&\BB_j
=
(\BX^{(n+1)})^T(\BD_j+\BSigma_t)\BX^{(n+1)},
\quad
\BM_s
=
(\BX^{(n+1)})^T\BSigma_s\BX^{(n+1)},\\
\text{and}\;&\overline{\BL}
=
\sum_{j=1}^{N_\Omega}
\omega_j\BL_j=
\Big(
\BI
-
\sum_{j=1}^{N_\Omega}
\omega_j \BB_j^{-1}\BM_s
\Big)^{-1}
\Big(
\sum_{j=1}^{N_\Omega}
\omega_j \BB_j^{-1}\bff_j
\Big).\label{eq:reduced-integral-rte}
\end{align}
\label{eq:deim_row_equation}
\end{subequations}
In practice, we first solve \eqref{eq:reduced-integral-rte} to obtain
$\overline{\BL}$.
With $\overline{\BL}$ available , the coefficients $\BL_j$ can be
 independently solved for each angle.
Collecting $\BL_j^T$ as the rows of $\BL$, a quadrature-weighted
orthogonalization yields
\begin{equation}
\BL
=
\BV^{(n+1)}(\BS^{(n+1)})^T,
\qquad
(\BV^{(n+1)})^T
\pzc{\textrm{diag}(\omega_1,\dots,\omega_{N_{\BOmega}}})
\BV^{(n+1)}
=
\BI,
\label{eq:deim_weighted_orth}
\end{equation}
This gives the updated factors $\BX^{(n+1)},\BS^{(n+1)},\BV^{(n+1)}$. Equation~\eqref{eq:deim_weighted_orth} produces a quadrature-weighted orthonormal angular basis, whereas the optional truncation in Alg.~\ref{alg:lr-si} uses Euclidean orthonormality. This distinction leaves the basis-invariant reconstruction \eqref{eq:deim_column_sweep} unchanged but may alter the angles selected by DEIM.

\textbf{(iv) Stopping criteria.} The outer iterations stop when $\|\bphi^{(n)}-\bphi^{(n-1)}\|_\infty\leq \varepsilon_{\phi}$.

The DEIM-based method selects representative angles from the previous iterate’s angular basis $\BV^{(n)}$, avoiding additional residual-based sampling. This procedure reuses angular information across iterations but requires access to $\BV^{(n)}$ and relies on its ability to capture the next iterate’s angular structure.

\section{Adaptive accuracy control and hybrid angular sampling for low-rank SI-DSA}
\label{sec:hybrid-lr-si}
We first introduce adaptive accuracy control for the sweep-based SI based on randomized greedy sampling in Sec. \ref{sec:rss-review}, then combine it with the DEIM-based approach in Sec. \ref{sec:deim-lr-si}. 

\subsection{Adaptive accuracy control}
\label{sec:adaptive-relaxation}
As shown in Fig.~\ref{fig:adaptive-accuracy}, early iterates remain far from convergence. Consequently, the accuracy requirements for low-rank SI can be relaxed during early iterations to reduce computational cost.

\textbf{Adaptive accuracy control}. We introduce a relaxation factor  $\tau_n$ to determine the inner-loop thresholds for Alg. \ref{alg:lr-si} in the $n$-th outer-loop iteration. Define the smallest and the largest relaxation factor allowed as $\tau_{\min}$ and $\tau_{\max}$. Initializing $\tau_1=\tau_{\max}$,  $\tau_n$ is adaptively updated as follows:
\begin{enumerate}
    \item\textbf{Inner-loop low-rank SI and DSA correction.} Perform inner-loop rank-adaptive SI in Alg. \ref{alg:lr-si} with the inner-loop thresholds $\varepsilon_{\mathrm{res}}^{(n)}=\varepsilon_{\mathrm{diff}}^{(n)}=\tau_n$ to update the scalar flux $\bphi^{(n-\half)}$ defined below equation \eqref{eq:si-angular-update}. Then, we apply the DSA correction $\bphi^{(n)}=\bphi^{(n-\half)}+\delta\bphi^{(n)}$.

    \item \textbf{Update the relaxation factor.} Compute $R_n=\|\bphi^{(n)}-\bphi^{(n-1)}\|_\infty$, then update the relaxation factor as
    \begin{equation}
        \tau_{n+1}=\begin{cases}
            \tau_{\max}=\tau_1,\;\text{if}\; n=1,\\
            \max\{\tau_{\min},\min(\tau_{\max},\tau_nR_n/R_{n-1})\},\;\text{if}\;n\geq2.
        \end{cases}
        \label{eq:si-relaxation-update}
    \end{equation}
\end{enumerate} 

\textbf{Rationale.}
Interpreting low-rank SI-DSA as a preconditioned Richardson iteration, our residual-dependent inner accuracy relaxation follows the analysis of inexact Richardson iterations. 
As shown in \cite{kressner2011low,birken2015termination}, residual-relative perturbations can preserve convergence. Viewing the error due to low-rank compression as a residual perturbation, the following estimate makes this principle explicit.
\begin{theorem}[Inexact preconditioned Richardson estimate]
\label{thm:inexact-richardson}
Consider the inexact Richardson iteration for solving
$\BA\bx=\bb$, i.e.
\[
    \bx^{(n+1)}
    =
    \bx^{(n)}
    +
    \alpha\BM^{-1}
    \left(
        \widetilde{\bb}^{(n)}
        -
        \widetilde{\BA}^{(n)}
        \bx^{(n)}
    \right),
\]
where $\BM$ is a symmetric positive definite (SPD) preconditioner, $\alpha>0$ is a damping parameter, and the $\BM$-norm contraction factor $\rho$ satisfies
   $ \rho
    :=
    \left\|
        \BI-\alpha\BM^{-1}\BA
    \right\|_{\BM}
    <1.$
Define the exact residual, inexact residual, and residual perturbation by
\begin{equation}
    \br^{(n)}
    :=
    \bb-\BA\bx^{(n)},
    \qquad
    \widetilde{\br}^{(n)}
    :=
    \widetilde{\bb}^{(n)}
    -
    \widetilde{\BA}^{(n)}\bx^{(n)},
\end{equation}
If the relative residual perturbation satisfies
\begin{equation}
   \|
        \widetilde{\br}^{(n)}-\br^{(n)}
    \|_{\BM^{-1}}/\|\br^{(n)}\|_{\BM^{-1}}
    \leq
    \eta^{(n)},
    \label{eq:relative-perturbation-condition}
\end{equation}
then the error at the $(n+1)$st iteration satisfies
\begin{equation}
\label{eq:inexact-richardson-one-step}
    \left\|
        \bx^{(n+1)}-\bx^*
    \right\|_{\BM}
    \leq
    \left(
        \rho+\alpha K\eta^{(n)}
    \right)
    \left\|
        \bx^{(n)}-\bx^*
    \right\|_{\BM},
    \qquad
    K
    :=
    \left\|
        \BM^{-1}\BA
    \right\|_{\BM},
\end{equation}
where $\bx^*$ denotes the exact solution. In particular, if
$\eta^{(n)}
    \leq
    \overline{\eta}
    =
    \frac{1-\rho}{\alpha K}$,
then the inexact iteration remains contractive.
\end{theorem}

\begin{proof}
See Appendix~\ref{app:proof-inexact-richardson}.
\end{proof}
To specialize Theorem~\ref{thm:inexact-richardson} to SI--DSA, we take $\BA=\BI-\BT\BSigma_s$, $\alpha=1$, and $\BM=\BM_{\mathrm{DSA}}$ from \eqref{eq:dsa-preconditioner}. {
With suitable consistent discretization, $\BM_{\mathrm{DSA}}$ is SPD.} Under the theorem's assumptions, the analysis motivates controlling the low-rank perturbation relative to the outer residual. If the inner thresholds ensure $\|\widetilde{\br}^{(n)}-\br^{(n)}\|_{\BM^{-1}}\lesssim\tau_{n+1}$, then the condition \eqref{eq:relative-perturbation-condition} motivates
\begin{equation}
\label{eq:richardson-error-control}
\tau_{n+1}\lesssim
\overline{\eta}\|\br^{(n)}\|_{\BM^{-1}}.
\end{equation}
{
For exact SI--DSA, the Richardson formulation \eqref{eq:si-richardson} implies
$\|\bphi^{(n)}-\bphi^{(n-1)}\|_{\BM}
=\|\br^{(n-1)}\|_{\BM^{-1}}$. This identity motivates using successive iterate differences to
monitor the residual.
To avoid an additional preconditioner application in norm computation, we empirically use the inexpensive indicator $R_n=\|\bphi^{(n)}-\bphi^{(n-1)}\|_\infty$ as a surrogate for the residual norm. $R_n$ differs from $\|\br^{(n-1)}\|_{\BM^{-1}}$  in two ways: the $\BM^{-1}$-norm is replaced by the cheaper $\infty$-norm, and the residual is lagged by one iteration.
In addition, since the constant $\bar \eta$ in \eqref{eq:richardson-error-control} is not directly available, we cannot enforce that sufficient condition directly. Instead, the proposed adaptive accuracy control
\eqref{eq:si-relaxation-update} scales the inner tolerance with $R_n$. If $R_1>0$ and no clipping occurs after initialization, the update gives
\[\tau_{n+1}=\frac{\tau_2}{R_1}R_n=\frac{\tau_1}{R_1}R_n,\qquad n\geq2.\]
Thus, $\tau_{n+1}/R_n$ remains fixed: the inner tolerance tracks changes in the convergence indicator. This rule mirrors the residual-relative scaling suggested by \eqref{eq:richardson-error-control}, but does not guarantee that its sufficient condition holds.
}

\begin{remark}[Existing inexact-iteration theory]
Sec.~2.2.1 of \cite{kressner2011low} shows that fixed truncation accuracy limits the attainable residual in low-rank Richardson iteration and suggests reducing the tolerance roughly in proportion to the current residual.
For inexact fixed-point iterations, \cite{birken2015termination} reports that scaling the inner stopping tolerance by the outer residual generally yields a convergent iteration.
Recently, \cite{bachmayr2026low} uses coercivity and soft-thresholding to adapt iteration accuracy in a low-rank Richardson method with an exponential-sums preconditioner for the second-order formulation of the RTE.
\end{remark}


\subsection{Hybrid angular sampling}
\label{sec:hybrid-deim-greedy}
Randomized greedy \cite{Guo2026HighlyER} and
DEIM-based \cite{haut2026efficient} angular sampling have complementary strengths. Randomized greedy sampling does not require a preceding angular
basis: it selects representative angles according to their residuals
during SI and therefore enable direct rank adaptivity. However, it
does not explicitly exploit the low-rank angular structure identified
in preceding outer iterations. In contrast, DEIM reuses the preceding angular basis. In the steady-state setting, however, no such basis is available at initialization, and substantial changes in the angular subspace during early iterations may reduce its effectiveness.



We propose exploiting the complementary strengths of both strategies through a two-stage approach. During early outer iterations, the low-rank SI with
randomized greedy sampling (Alg.~\ref{alg:lr-si}), together with the adaptive
accuracy relaxation in Sec.~\ref{sec:adaptive-relaxation}, constructs and
adapts the low-rank bases without requiring a prescribed basis or rank. As the outer iteration approaches convergence, successive low-rank solutions vary only mildly, making the preceding angular basis sufficiently representative of the next iterate. We then switch to the DEIM-based iteration in
Sec.~\ref{sec:deim-lr-si}, thereby avoiding residual evaluations of candidate angles required by the randomized greedy sampling.

We identify the switching point using the change between successive scalar-flux
iterates.  Specifically, we switch after the first iteration $n_{\mathrm{switch}}$ satisfying
\begin{equation}
\|\bphi^{(n)}-\bphi^{(n-1)}\|_{\infty}
        \leq
        \varepsilon_{\mathrm{switch}},
    \label{eq:hybrid-switch-iteration}
\end{equation}
where $\varepsilon_{\mathrm{switch}}$ is a prescribed switching tolerance. At the $(n_{\mathrm{switch}}+1)$-st iteration, 
we use the factors $\BPsi^{(n_{\mathrm{sw}})}\approx    \BX^{(n_{\mathrm{sw}})}
    \BS^{(n_{\mathrm{sw}})}
    (\BV^{(n_{\mathrm{sw}})})^T$
to initialize the DEIM-based method in
Sec.~\ref{sec:deim-lr-si}, which is then used for all subsequent outer iterations. With a suitable choice of $\varepsilon_{\mathrm{switch}}$, this
two-stage method provides additional acceleration over pure randomized
greedy sampling without sacrificing accuracy.  However, if
$\varepsilon_{\mathrm{switch}}$ is chosen too large, the method may switch
before the low-rank spaces have sufficiently stabilized, causing the
subsequent outer iterations to stagnate.

\begin{remark}[Why not DEIM-based initialization]
DEIM can provide initial samples for randomized greedy sampling in Alg. \ref{alg:lr-si}. However, we observe that this initialization provides marginal or even no practical computational savings, because it avoids only the inexpensive initial
candidate sampling and residual evaluations  while requiring transport sweeps
for the entire set of DEIM-selected angles.
\end{remark}

\section{GMRES-DSA for the integral scalar-flux  reformulation with low-rank matrix-vector products}
\label{sec:ta-lr-fgmres}
As shown in Fig.~\ref{fig:adaptive-accuracy} and
\cite{warsa2004krylov}, SI-DSA may converge slowly for multiscale problems
with highly diffusive regions and strong material discontinuities. Following
\cite{warsa2004krylov}, we address this limitation by applying
GMRES-DSA to the integral reformulation for scalar flux
\eqref{eq:memory-efficient-reformulation}. Our method uses sweep-based
low-rank transport solver in Alg. \ref{alg:lr-si} to approximate the right-hand side and matrix--vector products. By confining low-rank approximations to these operations, the method
does not truncate the scalar-flux Krylov vectors and preserves their
orthogonality, making established inexact-GMRES analysis relying on this orthogonality
\cite{simoncini2003theory} applicable.
Sec.~\ref{sec:lr-matvec} introduces the low-rank matrix--vector product, and
Sec.~\ref{sec:gmres-theory} integrates it into flexible GMRES (FGMRES) with adaptive
accuracy control.

\subsection{Low-rank matvecs and right-hand side evaluation for the scalar flux reformulation}
\label{sec:lr-matvec}
The integral reformulation for scalar flux \eqref{eq:memory-efficient-reformulation} reads
\begin{equation}
    (\BI-\BT\BSigma_s)\bphi=\bb,
\;
\BT=\sum_{j=1}^{N_\Omega}\omega_j(\BD_j+\BSigma_t)^{-1},
\;
\bb=\sum_{j=1}^{N_\Omega}\omega_j(\BD_j+\BSigma_t)^{-1}(\bG+\bg_j^{\mathrm{bc}}).
\label{eq:integral-scalar-flux-def}
\end{equation}
For a given vector $\bz$, evaluating
$(\BI-\BT\BSigma_s)\bz$ requires transport sweeps over all
angles with the scattering source $\BSigma_s\bz$, followed by angular
quadrature. In our low-rank setting, we instead employ the rank-adaptive sweep-based solver via randomized greedy sampling 
(Alg.~\ref{alg:lr-si}). Specifically, we first compute $\bs:=\BSigma_s \bz$ and pass $\bs$ to the low-rank transport solver as an isotropic  source. The
low-rank solver then adaptively samples representative angles and
constructs the approximation $\BPsi\approx\BX\BC$ with $\BPsi\in\mathbb{R}^{N_{\bx}\times N_\Omega}$, $\BX\in\mathbb{R}^{N_{\bx}\times r}$, and $\BC\in\mathbb{R}^{r\times N_\Omega}$. The associated angular integration is approximated by contraction
$\bphi=\BX(\BC\bw).$

In summary, the matrix--vector product $(\BI-\BT\BSigma_s)\bz$ is approximated by low-rank transport solves: 
\begin{equation}
    \BA_{\mathrm{lr}}(\bz;\tau)
    :=
    \bz-\widetilde{\BT}_{\mathrm{lr}}(\BSigma_s\bz;\tau),
    \label{eq:lr-matvec-operator}
\end{equation}
where $\widetilde{\BT}_{\mathrm{lr}}(\,\cdot\,;\tau)$ denotes one call to Alg.~\ref{alg:lr-si} with isotropic source $\BSigma_s\bz$ and inner-loop thresholds $\varepsilon_{\mathrm{res}}=\varepsilon_{\mathrm{diff}}=\tau$,  followed by the angular quadrature through contraction.

In the integral  formulation for scalar flux, transport sweeps over all angles are also required to construct the right-hand side $\bb=\sum_{j=1}^{N_\Omega}\omega_j(\BD_j+\BSigma_t)^{-1}(\bG+\bg_j^{\mathrm{bc}})$. We similarly accelerate this construction by applying the sweep-based
low-rank solver to the angle-dependent source
$\bG+\bg_j^{\mathrm{bc}}$ to produce a perturbed right-hand side $\widetilde{\bb}_{\mathrm{lr}}$.

In the above low-rank matrix-vector
products, we intentionally adopt randomized greedy angular sampling.
Unlike successive SI iterates near convergence, the input vectors in GMRES correspond to distinct preconditioned Arnoldi directions and may represent
different combinations of the spectral components of the operator
$\BI-\BT\BSigma_s$, rather than successive approximations to the same solution. Consequently, the angular basis constructed for one matrix--vector product may not remain representative of subsequent Krylov iterations, limiting the effectiveness of DEIM-based sampling. Randomized greedy sampling instead adapts the selected angles to each
matrix--vector product guided by its own residual.

\subsection{Inexact FGMRES with adaptively relaxed low-rank matvecs}
\label{sec:gmres-theory}
We integrate the low-rank matrix--vector products from
Sec.~\ref{sec:lr-matvec} into a flexible GMRES (FGMRES) solver for the integral scalar-flux reformulation. To improve efficiency, we adaptively control
the accuracy of low-rank products based on the GMRES residual,
following the heuristic strategy guided by the inexact-GMRES analysis in
\cite{simoncini2003theory}.

\textbf{Adaptive accuracy relaxation within the low-rank matrix--vector product.}
As in low-rank SI--DSA, we use a relaxation factor
$\tau_k\in[\tau_{\min},\tau_{\max}]$ to determine the inner-loop
thresholds at the $k$-th GMRES iteration.
Opposite to SI--DSA (Sec.~\ref{sec:adaptive-relaxation}), the residual-gap
analysis below motivates tight initial accuracy followed by relaxation
as the residual decreases.
Accordingly, we initialize $\tau_1=\tau_{\min}$ and update
\begin{equation}
    \tau_k
    =
    \min\left\{
        \tau_{\max},
        \frac{\tau_{\min}}
        {\max\left\{
            \|\br^{(k-1)}\|_2/\|\widetilde{\bb}\|_2,
            \tau_{\min}
        \right\}}
    \right\},
    \qquad k\geq 2,
    \label{eq:gmres-relaxation-rule}
\end{equation}
where $\br^{(k-1)}$ is the residual after the $(k-1)$-st GMRES
iteration.
The stopping thresholds for the randomized greedy low-rank solver
in Sec.~\ref{sec:rss-review} are set to
$\varepsilon_{\mathrm{res}}^{(k)}
=\varepsilon_{\mathrm{diff}}^{(k)}=\tau_k$.
The resulting inexact GMRES--DSA method is summarized in
Alg.~\ref{alg:ta-lr-fgmres}.

\begin{algorithm}[h]
\caption{Inexact FGMRES-DSA with adaptively relaxed low-rank matvecs.\label{alg:ta-lr-fgmres}}
\begin{algorithmic}[1]
\STATE{\textbf{Input}: Initial guess $\bphi^{(0)}$; approximate right-hand side
$\widetilde{\bb}$; right DSA preconditioner $\BM_{\mathrm{DSA}}^{-1}=\BI+\BA_{\mathrm{diff}}^{-1}\BSigma_s$ from \eqref{eq:dsa-preconditioner};
low-rank matrix--vector multiplication
$\BA_{\mathrm{lr}}(\,\cdot\,;\tau)$ from \eqref{eq:lr-matvec-operator}, with
inner-loop thresholds driven by $\tau$; maximum and minimum 
thresholds $\tau_{\max}$ and $\tau_{\min}$; relative tolerance 
$\varepsilon_{\mathrm{GMRES}}$; maximum number of iterations $m$.}
\STATE{\textbf{Initialization}: Compute the initial residual 
$\br^{(0)}=\widetilde{\bb}-\BA_{\mathrm{lr}}(\bphi^{(0)};\tau_{\min})$. 
Set $\beta=\|\br^{(0)}\|_2$, $\bq^{(1)}=\br^{(0)}/\beta$,
and $\tau_1=\tau_{\min}$.}
\FOR{$k=1,2,\dots,m$}
    \STATE{\textbf{Preconditioning and inexact matrix--vector product:}
    Compute
    $\bz^{(k)}=\BM_{\mathrm{DSA}}^{-1}\bq^{(k)}$ and
    $\bu^{(k)}=\BA_{\mathrm{lr}}(\bz^{(k)};\tau_k)$.}

    \STATE{\textbf{Arnoldi orthogonalization (modified Gram--Schmidt):}}
    \FOR{$i=1,2,\dots,k$}
        \STATE{Compute
        $h_{i,k}=(\bq^{(i)})^{\mathsf{T}}\bu^{(k)}$ and update
        $\bu^{(k)}=\bu^{(k)}-h_{i,k}\bq^{(i)}$.}
    \ENDFOR
    \STATE{Compute
    $h_{k+1,k}=\|\bu^{(k)}\|_2$ and
    $\bq^{(k+1)}=\bu^{(k)}/h_{k+1,k}$.}
    \STATE{Form the updated Hessenberg matrix
   $
        \overline{\BH}_k
        =
        \begin{bmatrix}
            \overline{\BH}_{k-1}
            &
            (h_{1,k},\ldots,h_{k,k})^{\mathsf{T}}
            \\
            \mathbf{0}^{\mathsf{T}}
            &
            h_{k+1,k}
        \end{bmatrix}.
    $}

    \STATE{\textbf{Residual minimization:}
    Apply Givens rotations to the last column of
    $\overline{\BH}_k$ and solve
        $\by^{(k)}
        =
        \arg\min_{\by\in\mathbb{R}^{k}}
        \left\|
            \beta\be_1-\overline{\BH}_k\by
        \right\|_2,$ and compute the residual norm as $\|\br^{(k)}\|_2
        =
        \left\|
            \beta\be_1-\overline{\BH}_k\by^{(k)}
        \right\|_2$,
    where $\be_1=(1,0,\dots,0)^T$.}

    \IF{$\|\br^{(k)}\|_2/
        \|\widetilde{\bb}\|_2
        \leq\varepsilon_{\mathrm{GMRES}}$}
        \STATE{\textbf{break}.}
    \ENDIF

    \STATE{\textbf{Relaxation update:}
        $\tau_{k+1}
        =
        \min\left\{
            \tau_{\max},
            \frac{\tau_{\min}}
            {\max\left\{
                \|\br^{(k)}\|_2/
                \|\widetilde{\bb}\|_2,
                \tau_{\min}
            \right\}}
        \right\}.$}
\ENDFOR
\STATE{\textbf{Solution update}: Solve the least-squares problem
$\by^{(k)}\gets\arg\min_{\by}\|\beta\be_1-\overline{\BH}_k\by\|_2$ and update
$\bphi\gets\bphi^{(0)}+\BZ_k\by^{(k)}$, where 
$\BZ_k=[\bz^{(1)},\dots,\bz^{(k)}]$.}
\STATE{\textbf{Output}: Approximate scalar flux $\bphi$.}
\end{algorithmic}
\end{algorithm}

\textbf{Theoretical foundation.} 
The threshold-control rule \eqref{eq:integral-scalar-flux-def} in Alg.~\ref{alg:ta-lr-fgmres}  is motivated by the residual-gap analysis of inexact GMRES~\cite{simoncini2003theory}. When matrix--vector products are computed inexactly, the residual obtained from the Arnoldi recurrence deviates from the true residual by a gap that depends on the accumulated errors in these products. The following theorem extends Theorem 5.1 of \cite{simoncini2003theory} to FGMRES and quantifies this relation in the presence of a perturbed right-hand side. The proof follows the analysis in \cite{simoncini2003theory} and is deferred to Appx.~\ref{app:gmres-proof}. 

\begin{theorem}[\textit{Residual-gap estimate for inexact FGMRES with an inexact right-hand side}]
\label{thm:residual-gap}
Consider the exact system $\BA\bphi=\bb$, with an approximate
right-hand side $\widetilde{\bb}$ and initial matrix--vector product 
satisfying
\begin{equation}
    \|\bb-\widetilde{\bb}\|_2\leq\varepsilon_0,
    \qquad
    \|\BA\bphi^{(0)}
    -\BA_{\mathrm{lr}}(\bphi^{(0)};\tau_0)\|_2
    \leq\varepsilon_0.
\end{equation}
The computed initial residual is defined by
\[
\br^{(0)}
:=
\widetilde{\bb}-\BA_{\mathrm{lr}}(\bphi^{(0)};\tau_0),
\qquad
\beta:=\|\br^{(0)}\|_2>0,
\qquad
\bq^{(1)}:=\br^{(0)}/\beta.
\]
Suppose that $m$ steps of the inexact flexible Arnoldi process, starting from
$\bq^{(1)}$, produce
\begin{equation}
\BA\BZ_m+\BE_m=\BQ_{m+1}\overline{\BH}_m,
\qquad
\BE_m=
\begin{bmatrix}
\bepsilon_1 & \cdots & \bepsilon_m
\end{bmatrix},
\label{eq:inexact_flexible_arnoldi}
\end{equation}
where $\BQ_{m+1}$ has orthonormal columns and
$\overline{\BH}_m\in\mathbb{R}^{(m+1)\times m}$ is upper Hessenberg
and has full column rank.
For $k=1,\ldots,m$, let $\BZ_k$ and $\BQ_{k+1}$ contain the first $k$
and $k+1$ columns of $\BZ_m$ and $\BQ_{m+1}$, respectively, and let
$\overline{\BH}_k
:=
\overline{\BH}_m(1\!:\!k+1,1\!:\!k).$
Define
\begin{equation}
\by^{(k)}
:=
\arg\min_{\by\in\mathbb{R}^k}
\left\|
\beta\be_1-\overline{\BH}_k\by
\right\|_2,
\qquad
\bphi^{(k)}
:=
\bphi^{(0)}+\BZ_k\by^{(k)},
\label{eq:fgmres_iterate}
\end{equation}
and define the corresponding computed residual by
\begin{equation}
\br^{(k)}
:=
\BQ_{k+1}
\left(
\beta\be_1-\overline{\BH}_k\by^{(k)}
\right).
\label{eq:computed_fgmres_residual}
\end{equation}
The true residual at the $m$-th iteration is
$\br_{\mathrm{true}}^{(m)}:=\bb-\BA\bphi^{(m)}$.
Let $\varepsilon_{\mathrm{gap}}>0$ be a prescribed residual-gap tolerance,
and assume $0\leq\varepsilon_0< \frac{1}{2}\varepsilon_{\textrm{gap}}$.
Assume that $\|\br^{(k-1)}\|_2>0$ for $k=1,\ldots,m$.
If the error in each inexact matrix--vector multiplication satisfies
\begin{equation}
\|\bepsilon_k\|_2
\leq
\frac{\sigma_m(\overline{\BH}_m)}{m}
\frac{
\varepsilon_{\mathrm{gap}}
-2\varepsilon_0
}{
\|\br^{(k-1)}\|_2
},
\qquad k=1,\ldots,m,
\label{eq:gmres-error-constraint}
\end{equation}
where $\sigma_m(\overline{\BH}_m)>0$ is the smallest singular value
of $\overline{\BH}_m$, then
$
\left\|
\br_{\mathrm{true}}^{(m)}-\br^{(m)}
\right\|_2
\leq
\varepsilon_{\mathrm{gap}}.
$
\end{theorem}

Under the assumptions of Theorem~\ref{thm:residual-gap}, the residual
gap is bounded by $\varepsilon_{\mathrm{gap}}$, so
$\|\br^{(m)}\|_2\leq\varepsilon$ implies
$\|\br_{\mathrm{true}}^{(m)}\|_2\leq\varepsilon+\varepsilon_{\mathrm{gap}}$.
Thus, the computed residual provides a reliable convergence indicator
up to the prescribed residual-gap tolerance.

The bound also motivates adaptive accuracy control of the inexact
matrix--vector products. As noted in \cite{simoncini2003theory},
the smallest singular value $\sigma_m(\overline{\BH}_m)$ in
\eqref{eq:gmres-error-constraint} is generally unavailable in practice. We therefore use the bound to inform the choice of thresholds. Assuming $\|\bepsilon_k\|_2\lesssim\tau_k$,
the error dependence in
\eqref{eq:gmres-error-constraint} motivates the empirical rule
\begin{equation}
\label{eq:tolerance-rule}
\tau_k \propto \frac{1}{\|\br^{(k-1)}\|_2},
\qquad
\varepsilon_{\mathrm{res}}^{(k)} \propto \tau_k,
\qquad
\varepsilon_{\mathrm{diff}}^{(k)} \propto \tau_k.
\end{equation}
As the computed residual decreases, this rule allows larger thresholds and hence cheaper inexact matrix--vector products,
guiding the threshold relaxation in
Alg.~\ref{alg:ta-lr-fgmres}.

The theorem reserves $2\varepsilon_0$ for the relaxed and initial-product errors, leaving $\varepsilon_{\mathrm{gap}}-2\varepsilon_0$ for subsequent matrix--vector product errors. Hence, a tight threshold should be applied when constructing the
right-hand side and initial residual. For the zero initial guess used throughout
Sec.~\ref{sec:numerical}, $\br^{(0)}=\widetilde{\bb}$ and the
initial-product error vanishes, increasing the allowable error
 for subsequent iterations to
$\varepsilon_{\mathrm{gap}}-\varepsilon_0$.

Crucially, the relaxed approximation and inexact matrix--vector products
are the only sources of low-rank approximation error in our method.
Unlike low-rank Krylov methods that also truncate basis vectors or
their orthogonalization updates \cite{kressner2011low,dolgov2013tt},
our method does not compress the scalar-flux Krylov basis.
Such truncations can perturb the Arnoldi relation and compromise
orthogonality, potentially requiring additional error control and
reorthogonalization \cite{palitta2021convergence}.
This orthonormality is used both in the analysis of
\cite{simoncini2003theory} and in the proof of
Theorem~\ref{thm:residual-gap} (Appendix~\ref{app:gmres-proof})
to show that the projected least-squares residual norm equals
the computed residual norm.

\section{Numerical examples\label{sec:numerical}}

We evaluate the accuracy, efficiency, and robustness of the proposed
low-rank solvers on a series of benchmark problems. The methods considered
are low-rank SI--DSA (LR--SI--DSA) and FGMRES--DSA with low-rank
matrix--vector products (LR--GMRES--DSA). Throughout this section, LR and FR
denote low-rank and full-rank computations, respectively.

For each comparison, the LR and FR variants use identical DSA
preconditioners and zero initial guesses. A fully-consistent DSA preconditioner \cite{Adams2002FastIM} is applied. The diffusion-correction equation in DSA is solved using conjugate gradient method
preconditioned by a V-cycle algebraic multigrid (AMG) with a relative residual
tolerance of $10^{-12}$. The AMG is implemented  based on the \texttt{iFEM} package \cite{chen2009integrated}. To enable a direct comparison
between SI--DSA and GMRES--DSA, we replace the relative-residual criterion
specified for GMRES in Alg.~\ref{alg:ta-lr-fgmres} with a common
scalar-flux stopping criterion. Both solvers terminate when
\[
\|\bphi^{(n)}-\bphi^{(n-1)}\|_{\infty}
\leq \varepsilon_{\phi},
\qquad \varepsilon_{\phi}=10^{-6}.
\]
Unless otherwise stated, the adaptive thresholds are bounded by
$\tau_{\min}=10^{-7}$ and $\tau_{\max}=10^{-2}$. Fixed-accuracy low-rank methods
use $\tau_{\min}$ throughout, whereas adaptive low-rank approaches control
the thresholds from $[\tau_{\min},\tau_{\max}]$ according to the strategies
developed in the preceding sections. 
 
For each test, we report the number of source or Krylov iterations, wall-clock
time, speedup, total number of sweeps, and errors relative to the corresponding FR solution. The
speedup is defined by
$\text{speedup}=\frac{t_{\mathrm{FR}}}{t_{\mathrm{LR}}}$,
where $t_{\mathrm{FR}}$ and $t_{\mathrm{LR}}$ are the wall-clock times of the
corresponding FR and LR solvers. For FGMRES--DSA, the reported sweep
counts include the sweeps required to construct the right-hand side. Thus,
for FR--GMRES--DSA with zero initial guess,
\[
\text{Total number of sweeps}=(\text{Number of iterations}+1)\times N_\Omega.
\]
To assess sensitivity to randomized
sampling, LR results are reported as the mean $\pm$ standard deviation
over ten independent random seeds.

All methods are implemented in {\tt Matlab R2025b}. All the numerical tests are performed on a 2025 MacBook Pro laptop with a M4 chip and 24~GB  memory.

\subsection{Homogeneous medium in the diffusion regime}
\label{sec:homo}

We consider the computational domain $[-1,1]^2$ with no absorption,
a constant scattering $\sigma_s=100$, and zero inflow
boundary conditions. An isotropic Gaussian source is imposed as
\begin{equation}
    G(x,y)=\exp\bigl(-100(x^2+y^2)\bigr).
\end{equation}
The spatial domain is discretized using an $80\times80$ rectangular
mesh, while the CL$(40,20)$ angular quadrature is applied. 

Due to strong scattering, this problem is close to its diffusion limit, thereby admits efficient low-rank compression. We compare the full-rank, fixed-accuracy low-rank, and
accuracy-adaptive low-rank SI--DSA and FGMRES--DSA. As mentioned, each method is tested repeatedly against ten random seed to test their robustness. The
results are summarized in
Tab.~\ref{tab:case1_results}. All low-rank solvers achieve an error on the order of $10^{-7}$, a magnitude smaller than the outer-loop stopping threshold $\varepsilon_{\phi}=10^{-6}$. 

\textbf{LR--SI--DSA.} Due to the high compressibility of this problem, fixed-accuracy
LR--SI--DSA is already efficient. It reduces the total number of sweeps by a factor of approximately $13$ and achieves a mean speedup of $7.15\times$.
Adaptive accuracy control reduces the number of sampled angles and hence the transport-sweep cost of each inner solve. 
The loose initial threshold $\tau_{\max}=10^{-2}$ reduces the mean sweeps per
outer iteration from $60.1$ for fixed accuracy to $32.3$ but requires $1.78$ times as many outer
iterations. Despite the larger outer iteration count,
the total number of sweeps decreases slightly from approximately $307$ to $294$,
resulting in a speedup of $5.30\times$.
With more conservative maximum thresholds, $\tau_{\max}=10^{-3}$ and $10^{-4}$,  low-rank SI converges with outer iterations closer to full-rank SI, and reduces the total sweeps to approximately $239$ and $230$, respectively. In particular, with
$\tau_{\max}=10^{-4}$, the adaptive accuracy control leads to a speedup of
$8.02\times$ slightly higher than the fixed-accuracy variant.

\textbf{LR--FGMRES--DSA.}
LR--FGMRES--DSA is considerably more robust to aggressive relaxation. With
$\tau_{\max}=10^{-2}$, all runs retain the same $6$ Krylov iterations as the
fixed-accuracy and full-rank methods, while the mean number of transport
sweeps decreases from approximately $5600$ in full-rank methods to $771.2$ with fixed accuracy and $288.9$ with adaptive relaxation. 
The adaptively relaxed method is therefore $2.06\times$ faster than its
fixed-accuracy counterpart and $8.78\times$ faster than the full-rank FGMRES-DSA,
without increasing the scalar-flux error.

\textbf{Comparison.}
Overall, adaptive accuracy control accelerates the low-rank solvers by reducing
the number of angles sampled during the iterations. In this test,
FGMRES--DSA tolerates aggressive relaxation without additional Krylov
iterations, whereas adaptively relaxed low-rank SI--DSA requires a tighter upper threshold to preserve
its outer convergence iteration counts. For this diffusive homogeneous problem, SI--DSA converges rapidly and avoids the additional sweeps required to construct the FGMRES right-hand side, making its best adaptive variant slightly faster.

\begin{table}[!htbp]
\centering
\scriptsize
\textup{(a) SI--DSA with different maximum relaxation factors.}
\vspace{1mm}
\begin{tabular}{c l c c c c c}
\toprule
$\tau_{\max}$
& Method
& Time (s)
& Iter.
& Sweeps 
& Rel. Err. $\phi$
&  Speedup \\
\midrule

--
& FR
& 130.78
& 5
& 4000
& --
& 1.00 \\

--
& LR--F
& 18.30 $\pm$ 1.31
& 5.1 $\pm$ 0.3
& 306.5 $\pm$ 19.3
& $(1.68\pm0.69)\times10^{-7}$
& 7.15 \\

\midrule

$10^{-2}$
& LR--A
& $24.66 \pm 5.14$
& $9.1 \pm 1.3$
& $294.0 \pm 77.0$
& $(1.56\pm0.46)\times10^{-7}$
& 5.30 \\

$10^{-3}$
& LR--A
& $17.83 \pm 1.91$
& $6.2 \pm 0.4$
& $239.0 \pm 36.0$
& $(1.26\pm0.29)\times10^{-7}$
& 7.33 \\

$10^{-4}$
& LR--A
& \textbf{$16.30 \pm 1.28$}
& $5.1 \pm 0.3$
& $230.0 \pm 20.0$
& $(1.49\pm0.86)\times10^{-7}$
& \textbf{8.02} \\

\bottomrule
\end{tabular}%
\vspace{3mm}

\textup{(b) FGMRES with DSA preconditioning. Sweep counts include those used
for both relaxed construction and Krylov matrix--vector products.}

\vspace{1mm}
\begin{tabular}{c l c c  c c c}
\toprule
$\tau_{\max}$
& Method
& Time (s)
& Iter.
& Sweeps
& Rel. Err. $\phi$
& Speedup \\
\midrule

--
& FR
& 192.66
& 6
& 5600
& --
& 1.00 \\

--
& LR--F
& 45.26 $\pm$ 0.79
& 6.0 $\pm$ 0.0
& 771.2 $\pm$ 6.3
& $(2.61\pm1.97)\times10^{-7}$
& 4.26 \\

\midrule

$10^{-2}$
& LR--A
& \textbf{21.95 $\pm$ 0.64}
& 6.0 $\pm$ 0.0
& 288.9 $\pm$ 4.6
& $(2.63\pm1.94)\times10^{-7}$
& \textbf{8.78} \\

\bottomrule
\end{tabular}%
\caption{Results for the homogeneous diffusion-regime problem in
Sec.~\ref{sec:homo}. LR-F: low-rank methods with fixed inner-loop accuracy. LR-A: low-rank methods with adaptive inner-loop accuracy relaxation. Speedup represents the mean value.}
\label{tab:case1_results}
\end{table}

\subsection{Lattice problem}
\label{sec:lattice}

We next consider a lattice problem with zero inflow boundary conditions
on the computational domain $[0,5]^2$. The material configuration is
shown in Fig.~\ref{fig:lattice}. The black regions are purely
absorbing with $(\sigma_a,\sigma_s)=(100,0)$, whereas the remaining
region is purely scattering with $(\sigma_a,\sigma_s)=(0,1)$. A
constant source $G(x,y)=1$ is imposed in the central orange region. We use a uniform $80\times80$ rectangular spatial mesh
and the CL$(32,16)$ quadrature in angular space.
A reference full-rank scalar flux is shown in Fig.~\ref{fig:lattice}.

This example tests whether adaptive inner tolerances convert a looser outer
accuracy requirement into computational savings. The two solvers adapt their
inner accuracy in opposite directions. 
As suggested by our analysis, LR--SI starts from $\tau_{\max}=10^{-2}$ and tightens
toward the fixed bound $\tau_{\min}=0.1\varepsilon_\phi$, whereas LR--FGMRES starts from
$\tau_{\min}=0.1\varepsilon_\phi$ and relaxes toward
$\tau_{\max}=10^{-2}$ as the Krylov residual decreases. The results are
summarized in Tabs.~\ref{tab:lattice_si_cost} and
\ref{tab:lattice_fgmres_cost}.
Errors are measured against the full-rank SI-DSA reference computed with $\varepsilon_\phi=10^{-8}$.

\textbf{Accuracy and convergence.}
Across all tested outer tolerances, the fixed-accuracy and adaptive-accuracy low-rank methods closely track
the iteration counts of their full-rank counterparts. For SI--DSA, the
low-rank errors are essentially identical to the corresponding full-rank
errors, indicating that the total error is dominated by the prescribed outer stopping tolerance. For low-rank FGMRES--DSA, the error is larger than the full-rank counterpart but still comparable or below the prescribed outer-loop tolerance.
The small variations across random seeds
further demonstrate the robustness of our methods.

\textbf{Efficiency gains from relaxed outer tolerances.}
For LR--SI--DSA, as the outer convergence tolerance is
loosened, the fixed-accuracy variant provides nearly uniform speedups of $1.10$--$1.31\times$. Meanwhile, the speedup of the adaptive-accuracy
variant increases monotonically as the outer convergence tolerance is
loosened, from $1.88\times$ at $\varepsilon_\phi=10^{-6}$ to
$2.78\times$ at $\varepsilon_\phi=10^{-3}$.

The gain is more pronounced for LR-FGMRES--DSA.
As $\varepsilon_\phi$ is
relaxed from $10^{-6}$ to $10^{-3}$, the speedup of fixed-accuracy variant increases from
$1.30\times$ to $2.53\times$, while the speedup of the variant with adaptive accuracy control increases from
$2.13\times$ to $4.04\times$. Relative to the fixed-accuracy low-rank FGMRES, adaptive relaxation
provides an additional $1.59$--$1.76\times$ acceleration while preserving
the full-rank Krylov iteration counts. These results demonstrate that
adaptive accuracy control produces increasing computational savings as the desired outer accuracy is relaxed.

\textbf{Mechanism and robustness.}
Fig.~\ref{fig:lattice_adaptive_summary} illustrates the mechanisms behind
these gains. For LR--SI--DSA, the tolerance tightens from $10^{-2}$ toward
$10^{-7}$, and the effective rank correspondingly increases from
approximately $55$ to $225$--$230$. Adaptive accuracy control therefore postpones
high-rank computations until they become necessary. For LR--FGMRES--DSA,
consistent with Theorem~\ref{thm:residual-gap}, the right-hand side and early
matrix--vector products require high accuracy, whereas later products can be
relaxed and computed at lower rank as the Krylov residual decreases. The
LR--FGMRES and LR-SI residual histories closely follow the full-rank history, while the
standard-deviation bands remain almost invisible, further demonstrating the robustness against
randomized sampling.

\textbf{Comparison between LR-GMRES and LR-SI.}
Across all common outer tolerances
$\varepsilon_\phi=10^{-6}$--$10^{-3}$, LR--FGMRES with adaptive accuracy control consistently requires
fewer iterations and less computational time than LR--SI with adaptive accuracy control. It completes in
$4$--$6$ Krylov iterations, compared with $6$--$11$ outer SI iterations, and
is $1.70$--$1.89\times$ faster. Moreover, the speedup of LR-GMRES with adaptive accuracy control over the corresponding
full-rank solver increases from $2.13\times$ to $4.04\times$, compared with
$1.88\times$ to $2.78\times$ for LR-SI with adaptive accuracy control. Thus, adaptively relaxed LR--FGMRES is
consistently the most efficient method across all shared accuracy levels.

\textbf{Ray-effect mitigation.}
Another practical benefit of the low-rank formulation is that part of the
computational savings can be reinvested into a finer angular discretization.
For the full-rank SI-DSA with $\varepsilon_{\phi}=10^{-6}$ and CL($32$,$16$) quadrature rule, the runtime in Tab. \ref{tab:angular_resolution_summary} is $184.69$~s,
LR--A requires only $99.20\pm2.47$~s at the same angular resolution and outer convergence tolerance (Tab. \ref{tab:angular_resolution_summary}).
Increasing the angular resolution to CL($40$,$20$) quadrature rule, the wall clock time for LR-SI-DSA with adaptive accuracy control and the same outer tolerance becomes 
$147.66\pm8.52$~s, which is still lower than the FR-SI-DSA with
CL($32$,$16$) quadrature. As shown in Fig. \ref{fig:lattice_ray_effect}, the ray effects are slightly less severe using LR-SI-DSA with finer angular resolution. 

\textbf{Impact of angular resolutions.}
To investigate the impact of angular resolution on low-rank performance, we fix the spatial grid at $(N_x,N_y)=(80,80)$ and the outer convergence tolerance at $\varepsilon_{\phi}=10^{-6}$. Tab.~\ref{tab:angular_resolution_summary} summarizes the performance of LR--SI--DSA with adaptive accuracy control at various angular resolutions. 
Results in Tab.~\ref{tab:angular_resolution_summary} were obtained in a separate set of runs with $10$ random seeds from those in Tab.~\ref{tab:lattice_si_cost} and errors were computed against different reference solutions, hence runtimes and errors for the shared configuration do not perfectly match.
As angular resolution increases, the effective rank grows more slowly than the total number of angles, reducing the compression ratio, i.e. the ratio between the solution rank and the angle number,  from $0.500$ to $0.340$ and increasing the sweep reduction from $2.06$ to $4.54$. Overall, LR-SI-DSA with adaptive accuracy control achieves mean speedups of $1.55$--$2.19$ across all tested resolutions while preserving the angular accuracy of the full-rank solver. However, runtime speedups do not increase monotonically with angular resolution due to the growing rank $r$: low-rank solvers involve low-rank overheads that scale as $O(r^2)$ or $O(r^3)$.

\begin{figure}[!htbp]
    \centering
    \includegraphics[width=0.47\textwidth]{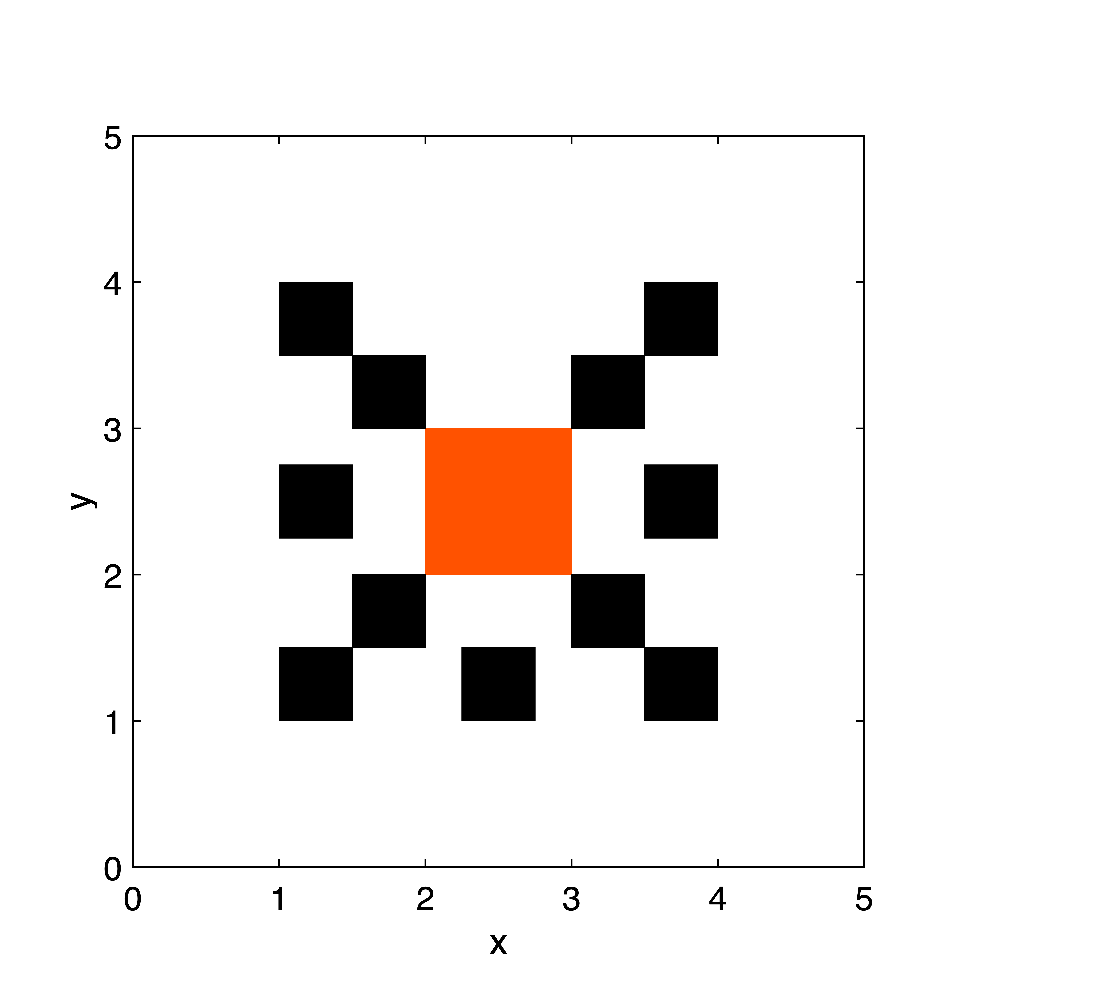}
    \includegraphics[width=0.47\textwidth]{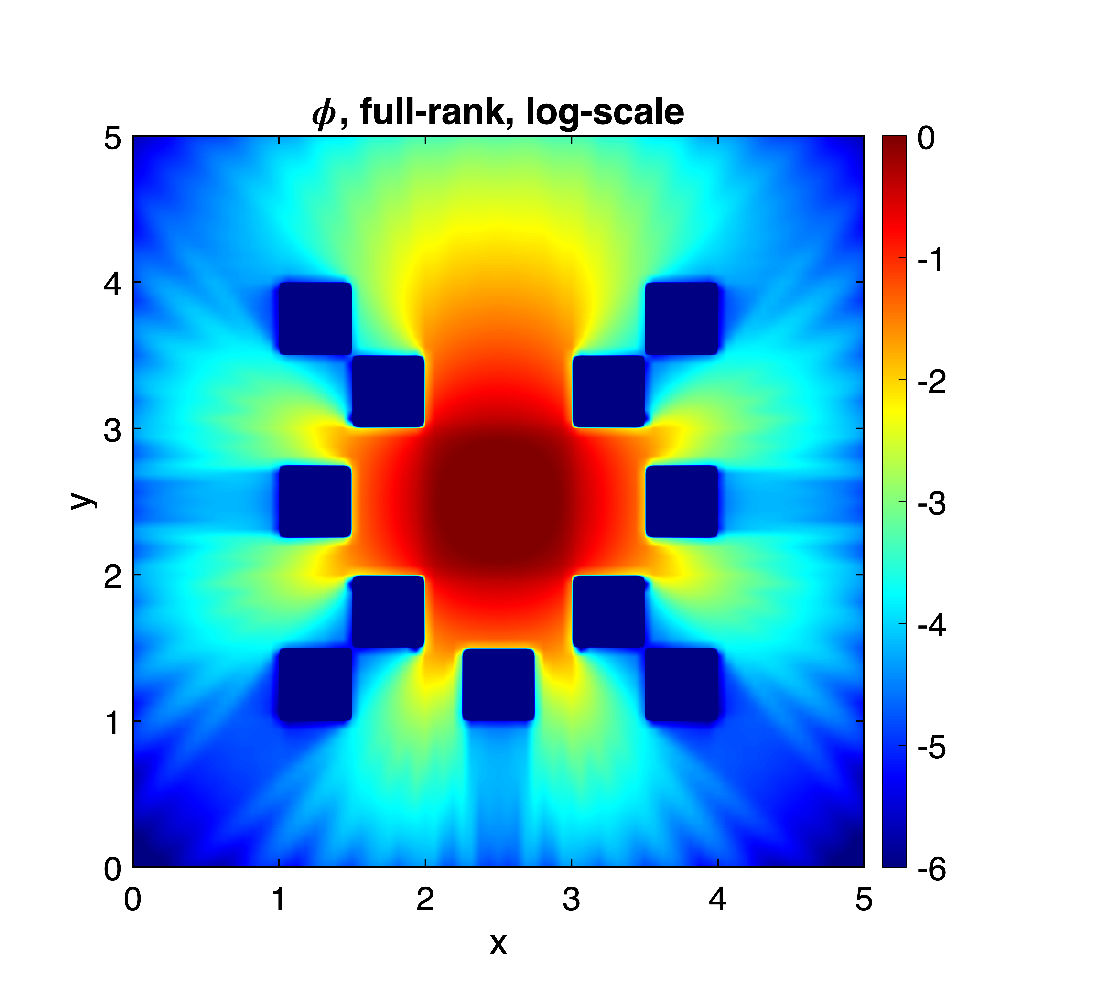}
        \caption{
Material configuration and scalar flux for the lattice problem in Sec.~\ref{sec:lattice}.
Left: material configuration.
Right: scalar flux obtained by FR--FGMRES--DSA on a logarithmic scale.
    \label{fig:lattice}
}
\end{figure}

\begin{figure}[!htbp]
    \centering
    \includegraphics[width=0.45\textwidth]{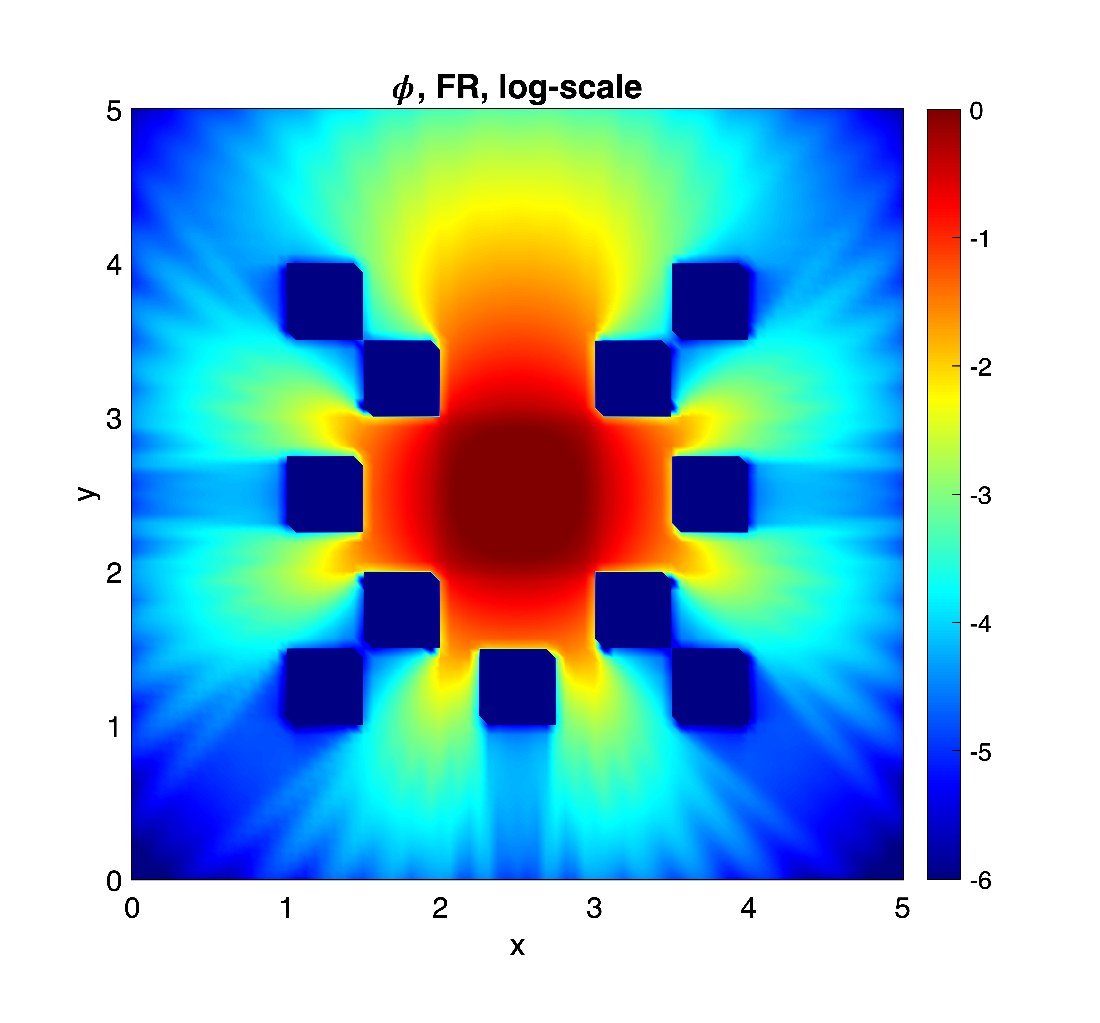}
    \includegraphics[width=0.45\textwidth]{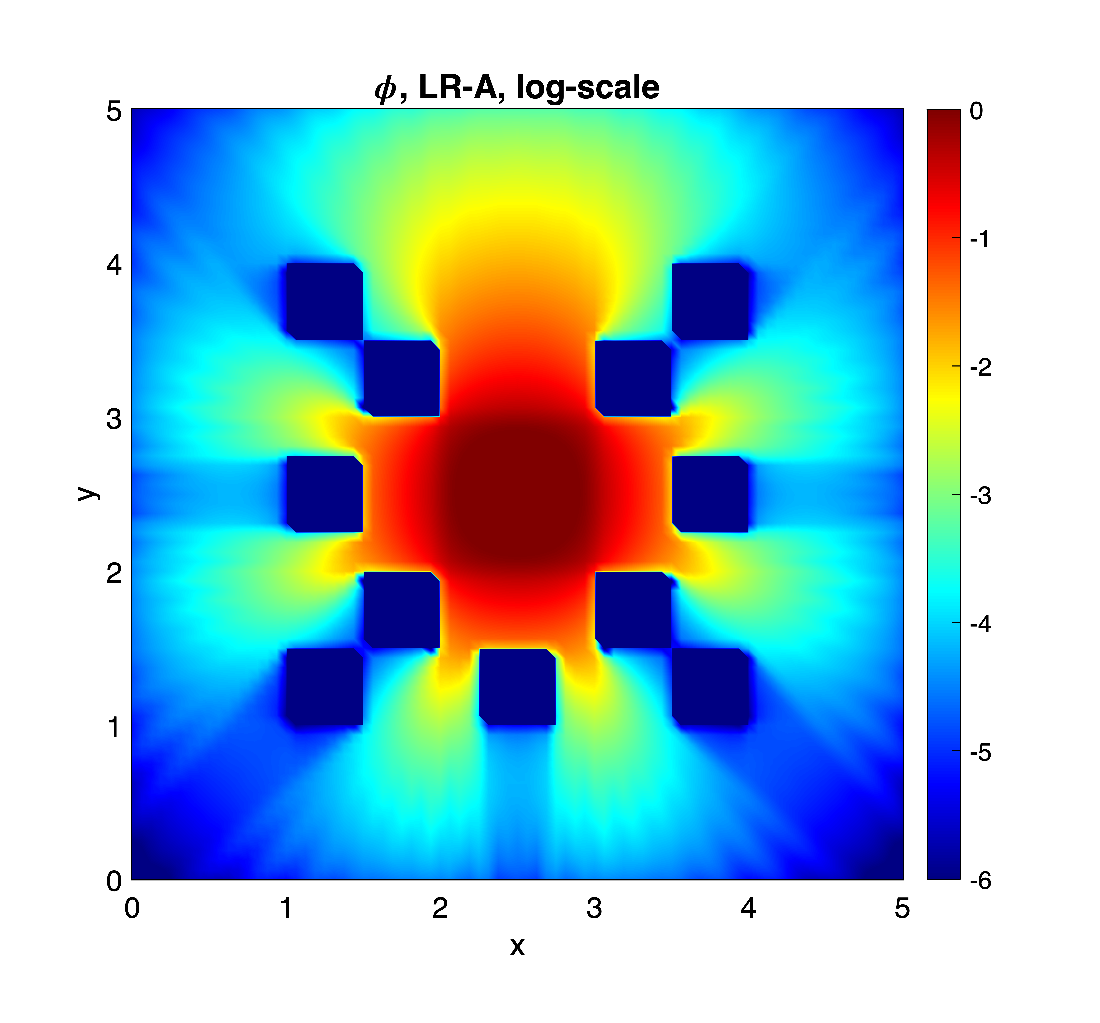}
    \caption{Scalar flux for the lattice problem in Sec. \ref{sec:lattice} using two angular resolutions. SI-DSA with $\varepsilon_{\phi}=10^{-6}$ is applied. Left: FR-SI-DSA with $CL(32,16)$ quadrature rule, total  computational time $184.69$ seconds (Tab. \ref{tab:angular_resolution_summary}). Right: LR-SI-DSA with adaptive accuracy control and $CL(40,20)$ quadrature rule, total  computational time $147.66\pm8.52$ seconds (mean $\pm$ std over ten random seeds). 
    Both solutions are shown on a logarithmic scale to highlight the effect of angular refinement on the directional transport features.}
    \label{fig:lattice_ray_effect}
\end{figure}

\begin{table}[!htbp]
\centering

\tiny

\begin{tabular}{c l c c c c c}
\toprule
$\varepsilon_\phi$
& Method
& Time (s)
& Iter.
& \shortstack{Sweeps}
& Rel. Err. $\phi$
& Speedup
\\
\midrule

\multirow{3}{*}{$10^{-6}$}
& FR
& 188.27
& 11
& 5632
& $1.44\times10^{-7}$
& 1.00
\\

& LR--F
& $171.50 \pm 15.77$
& $11.3 \pm 0.5$
& $2690.7 \pm 110.8$
& $1.60\times10^{-7}\pm1.23\times10^{-7}$
& 1.10
\\

& LR--A
& $\mathbf{100.16 \pm 2.31}$
& $11.0 \pm 0.0$
& $1788.7 \pm 22.8$
& $1.61\times10^{-7}\pm3.41\times10^{-8}$
& $\mathbf{1.88}$
\\

\midrule

\multirow{3}{*}{$10^{-5}$}
& FR
& 156.17
& 9
& 4608
& $2.52\times10^{-6}$
& 1.00
\\

& LR--F
& $125.26 \pm 2.11$
& $9.0 \pm 0.0$
& $2101.9 \pm 16.0$
& $2.53\times10^{-6}\pm1.29\times10^{-8}$
& 1.25
\\

& LR--A
& $\mathbf{73.38 \pm 3.97}$
& $9.1 \pm 0.3$
& $1338.4 \pm 69.0$
& $2.35\times10^{-6}\pm6.15\times10^{-7}$
& $\mathbf{2.13}$
\\

\midrule

\multirow{3}{*}{$10^{-4}$}
& FR
& 137.98
& 8
& 4096
& $1.05\times10^{-5}$
& 1.00
\\

& LR--F
& $110.68 \pm 4.14$
& $8.0 \pm 0.0$
& $1845.8 \pm 26.2$
& $1.05\times10^{-5}\pm4.95\times10^{-9}$
& 1.25
\\

& LR--A
& $\mathbf{59.31 \pm 1.02}$
& $8.0 \pm 0.0$
& $1093.6 \pm 13.4$
& $1.06\times10^{-5}\pm1.43\times10^{-7}$
& $\mathbf{2.33}$
\\

\midrule

\multirow{3}{*}{$10^{-3}$}
& FR
& 103.62
& 6
& 3072
& $1.82\times10^{-4}$
& 1.00
\\

& LR--F
& $79.04 \pm 1.95$
& $6.0 \pm 0.0$
& $1326.7 \pm 16.7$
& $1.82\times10^{-4}\pm4.39\times10^{-9}$
& 1.31
\\

& LR--A
& $\mathbf{37.22 \pm 0.57}$
& $6.0 \pm 0.0$
& $686.4 \pm 7.1$
& $1.82\times10^{-4}\pm1.42\times10^{-7}$
& $\mathbf{2.78}$
\\


\bottomrule
\end{tabular}%
\caption{Accuracy and computational performance of SI--DSA for the lattice
problem in Sec. \ref{sec:lattice} under different outer stopping tolerances $\varepsilon_\phi$.
LR--F: fixed-accuracy low-rank SI--DSA uses the fixed inner tolerance
$\tau_{\min}=0.1\varepsilon_\phi$. LR--A: adaptively tightens the inner tolerance from
$\tau_{\max}=10^{-2}$ toward $\tau_{\min}$. All errors are measured relative
to the same high-accuracy full-rank SI-DSA generated with
$\varepsilon_\phi=10^{-8}$. Speedup stands for the mean speedup.}
\label{tab:lattice_si_cost}
\end{table}

\begin{table}[!htbp]
\centering

\scriptsize

\begin{tabular}{c l c c c c c}
\toprule
$\varepsilon_\phi$
& Method
& Time (s)
& Iter.
& \shortstack{Sweeps}
& Rel. Err. $\phi$
& Speedup
\\
\midrule

\multirow{3}{*}{$10^{-6}$}
& FR & 117.04 & 6 & 3584
& $4.13\times10^{-8}$ & 1.00 \\
& LR--F & $89.94\pm2.87$ & $6.0\pm0.0$
& $1556.2\pm14.1$
& $(1.45\pm2.07)\times10^{-7}$ & 1.30 \\
& LR--A & $\mathbf{54.95\pm2.55}$ & $6.0\pm0.0$
& $907.6\pm20.5$
& $(3.20\pm1.64)\times10^{-7}$ & $\mathbf{2.13}$ \\

\midrule
\multirow{3}{*}{$10^{-5}$}
& FR & 118.64 & 6 & 3584
& $4.13\times10^{-8}$ & 1.00 \\
& LR--F & $75.95\pm1.17$ & $6.0\pm0.0$
& $1364.1\pm14.4$
& $(8.75\pm5.69)\times10^{-7}$ & 1.56 \\
& LR--A & $\mathbf{43.21\pm0.97}$ & $6.0\pm0.0$
& $712.0\pm12.3$
& $(1.95\pm0.45)\times10^{-6}$ & $\mathbf{2.75}$ \\
\midrule

\multirow{3}{*}{$10^{-4}$}
& FR & 101.74 & 5 & 3072
& $6.50\times10^{-7}$ & 1.00 \\
& LR--F & $54.46\pm1.22$ & $5.0\pm0.0$
& $990.5\pm13.4$
& $(5.50\pm3.02)\times10^{-6}$ & 1.87 \\
& LR--A & $\mathbf{31.45\pm0.73}$ & $5.0\pm0.0$
& $514.6\pm11.2$
& $(1.84\pm0.34)\times10^{-5}$ & $\mathbf{3.23}$ \\

\midrule

\multirow{3}{*}{$10^{-3}$}
& FR
& 83.97
& 4
& 2560
& $9.82\times10^{-6}$
& 1.00
\\

& LR--F
& $33.13 \pm 1.06$
& $4.0 \pm 0.0$
& $617.4 \pm 16.3$
& $(6.04\pm2.27)\times10^{-5}$
& 2.53
\\

& LR--A
& $\mathbf{20.78 \pm 0.57}$
& $4.0 \pm 0.0$
& $337.5 \pm 6.7$
& $(1.59\pm0.27)\times10^{-4}$
& $\mathbf{4.04}$
\\

\bottomrule
\end{tabular}%
\caption{Accuracy and computational performance of FGMRES--DSA for the
lattice problem in Sec. \ref{sec:lattice} under different outer stopping tolerances
$\varepsilon_\phi$. LR--F: fixed-accuracy low-rank FGMRES-DSA uses the fixed inner tolerance
$\tau_{\min}=0.1\varepsilon_\phi$. LR--A: adaptively relaxes the
inner tolerance from $\tau_{\min}$ toward $\tau_{\max}=10^{-2}$.
All errors are measured relative to the same high-accuracy full-rank
reference solution generated by SI-DSA with stopping criterion $\|\bphi^{(n)}-\bphi^{(n-1)}\|_\infty\leq\varepsilon_\phi=10^{-8}$. Speedup stands for the mean speedup. Sweep counts include those used
for both relaxed construction and Krylov matrix--vector products.}
\label{tab:lattice_fgmres_cost}
\end{table}

\begin{table}[!htbp]
\centering
\resizebox{\textwidth}{!}{
\begin{tabular}{llrrrrr}
\toprule
& Method
& Time (s)
& Speedup
& $r_{\mathrm{eff}}/N_{\BOmega}$
& Sweep red.
& Rel. Err. $\phi$ \\
\midrule

CL($16,8$)
& FR
& 91.89
& 1.00
& --
& --
& $1.81\times10^{-2}$ \\

& LR--A
& $42.01\pm1.54$
& $2.19$
& 0.50
& $2.06\times$
& $1.81\times10^{-2}\pm4.20\times10^{-9}$ \\

\midrule

CL($24,12$)
& FR
& 112.36
& 1.00
& --
& --
& $7.88\times10^{-3}$ \\

& LR--A
& $72.43\pm18.23$
& $1.55$
& 0.49
& $2.31\times$
& $7.88\times10^{-3}\pm7.40\times10^{-9}$ \\

\midrule

CL($32,16$)
& FR
& 184.69
& 1.00
& --
& --
& $3.66\times10^{-3}$ \\

& LR--A
& $99.20\pm2.47$
& $1.86$
& 0.44
& $3.15\times$
& $3.66\times10^{-3}\pm3.10\times10^{-10}$ \\

\midrule

CL($40,20$)
& FR
& 279.16
& 1.00
& --
& --
& $1.56\times10^{-3}$ \\

& LR--A
& $147.66\pm8.52$
& $1.89$
& 0.38
& $3.79\times$
& $1.56\times10^{-3}\pm2.40\times10^{-9}$ \\

\midrule

CL($48,24$)
& FR
& 398.99
& 1.00
& --
& --
& -- (ref.) \\

& LR--A
& $232.86\pm14.16$
& $1.71$
& 0.34
& $4.54\times$
& $(6.25\pm6.57)\times10^{-8}$ \\

\bottomrule
\end{tabular}
}
\caption{
Accuracy and computational performance of FR-SI-DSA and LR--SI--DSA with adaptive control for the lattice
problem in Sec.~\ref{sec:lattice} under increasing angular resolution.
We set the outer convergence tolerance as $\varepsilon_{\phi}=10^{-6}$, minimum and maximum accuracy relaxation factors as
$\tau_{\min}=10^{-7}$ and $\tau_{\max}=10^{-2}$.
The compression ratio $r/N_{\BOmega}$ is the effective solution rank
divided by the number of angles.
All errors are relative to the full-rank reference solution at the finest
angular resolution with CL($48,24$) quarature.
Speedup denotes the mean speedup, and Sweep red. denotes the reduction
factor in total transport sweeps relative to FR at the same angular resolution.
\label{tab:angular_resolution_summary}
}
\end{table}

\begin{figure}[!htbp]
    \centering
    \includegraphics[width=0.45\textwidth]{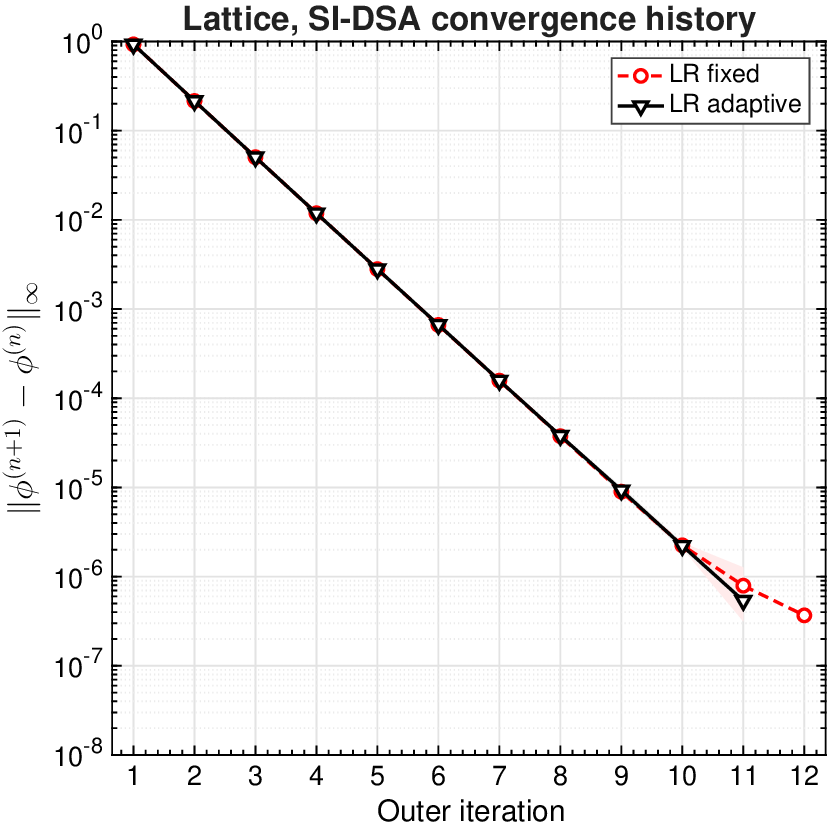}
    \includegraphics[width=0.45\textwidth]{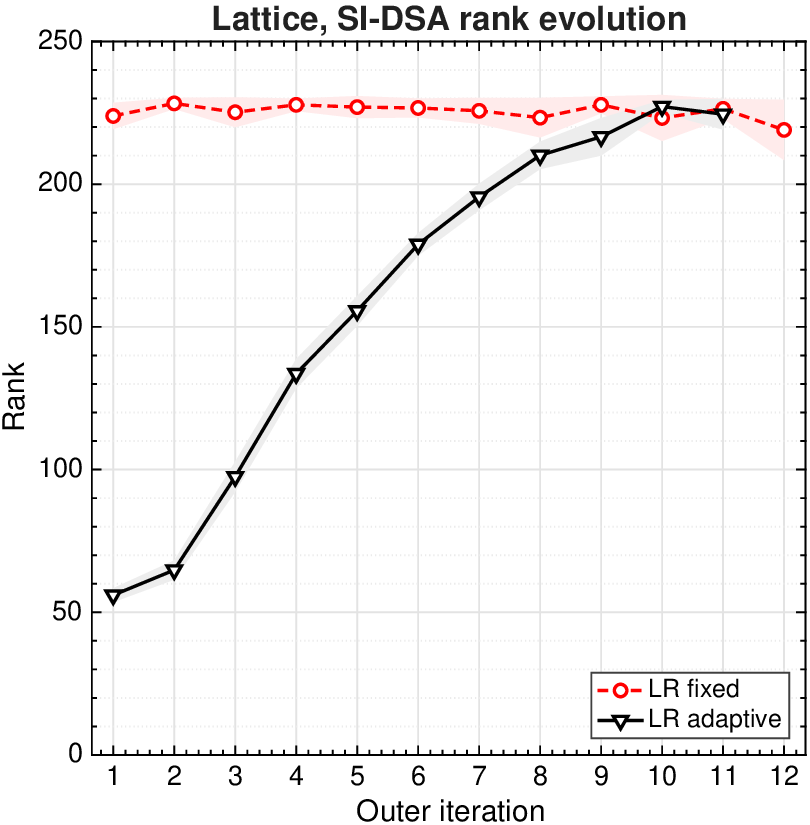}
    \includegraphics[width=0.45\textwidth]{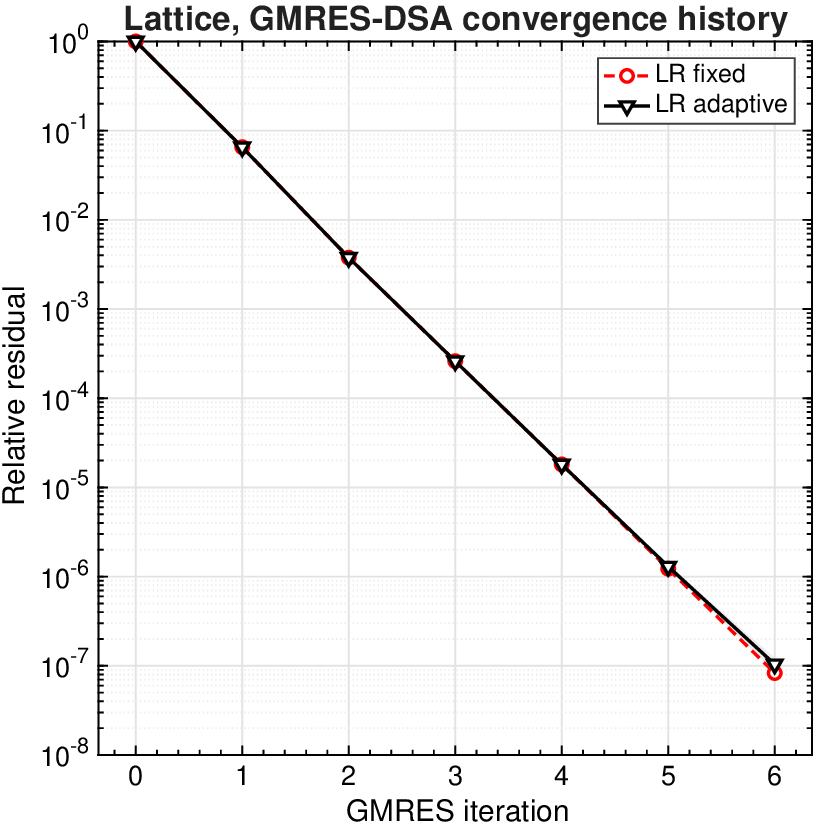}
    \includegraphics[width=0.45\textwidth]{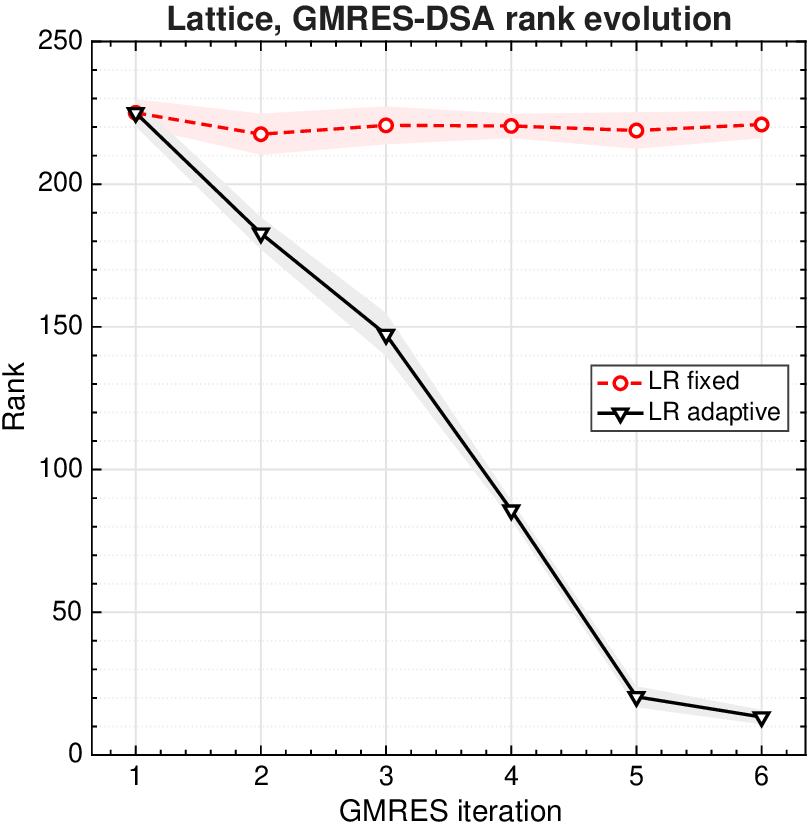}
      \caption{
Adaptive-accuracy behavior of LR--SI--DSA and LR--FGMRES--DSA
for the lattice problem with $\varepsilon_\phi=10^{-6}$.
Top left: convergence histories of fixed-accuracy and adaptive-accuracy
LR--SI--DSA. For all low-rank results, the plotted curves represent the mean over
ten random seeds, while the shaded regions indicate one standard
deviation about the mean.
Top right: corresponding effective-rank evolution.
Bottom left: relative-residual convergence histories of fixed-accuracy
and adaptive-accuracy LR--FGMRES--DSA.
Bottom right: corresponding effective-rank evolution.
}
\label{fig:lattice_adaptive_summary}
\end{figure}

\subsection{Pin-cell problem}
\label{sec:pin-cell}
We consider a  pin-cell problem to test the two-stage hybrid LR-SI-DSA combining DEIM and randomized greedy sampling, and compare it against the proposed LR-FGMRES-DSA. The computational domain is
$\mathcal{D}=[-1,1]^2$, with zero inflow boundary
conditions and no absorption. The source term and the scattering cross section are
defined as
\begin{equation}
G(x,y)=\exp(-100(x^2+y^2)),\quad
\sigma_s(x,y)
=
\begin{cases}
0.1,
& |x|\leq 0.5,\quad |y|\leq 0.5,
\\[2mm]
100,
& \text{otherwise}.
\end{cases}
\label{eq:pincell_sigma}
\end{equation}
Thus, the  central region is transport-dominated and surrounded by a diffusive medium whose scattering strength is $1000$ times that of the central region.
We use a uniform $80\times80$ Cartesian spatial mesh and the CL$(32,16)$ quadrature in angular space.  The material configuration and a full-rank reference scalar flux are presented in Fig. \ref{fig:pin-cell}. In this test, we  we set
$\tau_{\max}=10^{-3}$ and
 $\tau_{\min}=10^{-7}$ in the adaptive accuracy control.

\textbf{Two-stage hybrid LR-SI-DSA.}
Following Sec. \ref{sec:hybrid-deim-greedy}, we first apply LR-SI-DSA with adaptive accuracy test based on randomized greedy sampling in Alg. \ref{alg:lr-si}, and then switch to the
DEIM-based DLR solver in Sec. \ref{sec:deim-lr-si} once the prescribed switching criterion 
$\|\bphi^{(n)}-\bphi^{(n-1)}\|_{\infty}\leq\varepsilon_{\mathrm{switch}}$   is
satisfied.  We vary $\varepsilon_{\mathrm{switch}}$ over $10^{-2}$,
$5\times10^{-3}$, $5\times10^{-4}$, $10^{-4}$, and $10^{-5}$ and set
$\varepsilon_{\mathrm{DEIM-SI}}=10^{-6}$ in
\eqref{eq:inner-deim-stop}. The results are summarized in
Tab.~\ref{tab:pincell_hybrid}. 

With $\varepsilon_{\mathrm{switch}}=10^{-2},5\times10^{-3}$, the method switches after
only $3$ or $4$ randomized-greedy SI steps, before the angular basis becomes
sufficiently representative. As shown in Fig. \ref{fig:pin-cell}, the DEIM stage subsequently stagnates and fails
to reach the prescribed tolerance within the maximum number of outer
iterations. With $\varepsilon_{\mathrm{switch}}$  $=10^{-5}$, the efficiency of the hybrid method is slightly faster than the LR-SI-DSA with adaptive accuracy control but only randomized greedy sampling. By contrast, using $\varepsilon_{\mathrm{switch}}=5\times10^{-4},10^{-4}$ in the hybrid approach leads to additional accelerations over their counterparts with adaptive-accuracy control.

The convergence histories of the hybrid method in Fig. \ref{fig:pin-cell} exhibit a kink immediately after switching. It reflects a mismatch of the two sweep-based LR-SI:
randomized-greedy LR--SI produces an intermediate SI update (see line 4 of Alg. \ref{alg:lr-si}), whereas the
DEIM-based approach solves a perturbed intermediate $S_N$ system to its inner convergence tolerance (see equation \eqref{eq:deim_column_sweep}).

In summary, switching too early can substantially increase the SI iterations, while
switching too late limits the benefit of DEIM. The hybrid method with a properly chosen threshold,
balancing these effects, can achieve modest accelerations.

\textbf{Robustness and efficiency of adaptively relaxed LR--GMRES.}
Tab.~\ref{tab:pincell_fgmres} compares FR--GMRES with its fixed-accuracy and
adaptively relaxed low-rank variants. Both LR variants preserve the
$9$-iteration convergence of FR--GMRES for every random seed.
Fixed-accuracy LR--GMRES reduces the runtime from $171.84$ to $129.13$
seconds, yielding a $1.33\times$ speedup. Adaptive relaxation further reduces
the runtime to $70.65$ seconds, corresponding to speedups of $1.82\times$
over fixed-accuracy LR--GMRES and $2.43\times$ over FR--GMRES, while
maintaining scalar-flux errors of order $10^{-7}$.

\textbf{Comparison between LR-GMRES and LR-SI.}
As shown in Tabs.~\ref{tab:pincell_hybrid} and
\ref{tab:pincell_fgmres}, LR--GMRES with adaptive accuracy control requires $9$ Krylov
iterations and $70.65$ seconds, compared with $21.7$ outer iterations and
$241.18$ seconds for LR--SI with adaptive accuracy control. It is also approximately
$3.02\times$ faster than the best two-stage hybrid LR--SI result of
$213.48$ seconds. Thus, adaptively relaxed LR--GMRES is the most robust and
efficient method tested for this problem.

\begin{table}[!htbp]
\centering

\small
\renewcommand{\arraystretch}{1.08}

\resizebox{\textwidth}{!}{%
\begin{tabular}{l c c c c c c c c}
\toprule
&&&
\multicolumn{3}{c}{SI steps}
&&&
\\
\cmidrule(lr){4-6}
Method
& $\varepsilon_{\mathrm{switch}}$
& Time (s)
& RG
& DEIM
& Total
& Sweeps
& Rel. Err. $\phi$
& Speedup
\\
\midrule

FR
& --
& 362.86
& --
& --
& 21
& 10752
& --
& 1.00
\\

LR--F
& --
& $263.80 \pm 8.05$
& $21.2$
& --
& $21.2 \pm 0.6$
& $4574.2 \pm 124.4$
& $(1.04\pm1.05)\times10^{-7}$
& 1.38
\\

LR--A
& --
& $241.18 \pm 17.32$
& $21.7$
& --
& $21.7 \pm 1.3$
& $4193 \pm 289$
& $(1.55 \pm 1.24)\times 10^{-7}$
& 1.50
\\

\midrule

LR--H
& $10^{-2}$
& $935.64 \pm 96.87$
& 3.0
& Stagnate
& Stagnate
& --
& $(2.12\pm1.42)\times10^{-5}$
& --
\\

LR--H
& $5\times10^{-3}$
& $770.61 \pm 87.03$
& 4.0
& Stagnate
& Stagnate
& --
& $(6.32\pm7.01)\times10^{-6}$
& --
\\

LR--H
& $5\times10^{-4}$
& $238.52 \pm 40.26$
& 8.5
& 20.1
& $28.6 \pm 5.2$
& $5259.4 \pm 979.9$
& $(8.22\pm3.34)\times10^{-7}$
& 1.52
\\

LR--H
& $10^{-4}$
& $\mathbf{213.48 \pm 13.86}$
& 12.0
& 11.8
& $23.8 \pm 1.9$
& $4562.5 \pm 318.4$
& $(3.92\pm0.41)\times10^{-7}$
& $\mathbf{1.70}$
\\

LR--H
& $10^{-5}$
& $229.97 \pm 4.24$
& 17.0
& 6.4
& $23.4 \pm 0.8$
& $4522.2 \pm 172.9$
& $(3.87\pm0.18)\times10^{-7}$
& 1.58
\\

\bottomrule
\end{tabular}%
}
\caption{Results for the pin-cell problem using FR--SI--DSA and the proposed
low-rank variants. 
LR--A:  adaptively relaxed LR-SI-DSA with randomized greedy sampling.
LR--H: the
two-stage hybrid LR-SI-DSA.
For LR--H, ``RG'' and ``DEIM'' report the mean numbers of SI steps before and
after switching, respectively, while ``Total'' reports their sum as the mean
$\pm$ standard deviation.}
\label{tab:pincell_hybrid}
\end{table}
\begin{table}[!htbp]
\centering
\scriptsize
\begin{tabular}{l c c c c c}
\toprule
Method
& Time (s)
& Iter.
& Transport sweeps
& Rel. Err. $\phi$
& Mean Speedup \\
\midrule

FR
& 171.84
& 9
& 5120
& --
& 1.00 \\

LR--F
& 129.13 $\pm$ 3.74
& 9.0 $\pm$ 0.0
& 2214.6 $\pm$ 17.1
& $(2.09\pm1.58)\times10^{-7}$
& 1.33 \\

LR--A
& \textbf{70.65 $\pm$ 1.13}
& 9.0 $\pm$ 0.0
& 1226.2 $\pm$ 14.7
& $(4.58\pm2.07)\times10^{-7}$
& \textbf{2.43} \\
\bottomrule
\end{tabular}%
\caption{Results for the pin-cell problem in Sec.~\ref{sec:pin-cell} using
full-rank, fixed-accuracy low-rank, and adaptively relaxed
low-rank FGMRES--DSA. Sweep counts include those used
for both relaxed construction and Krylov matrix--vector products.}
\label{tab:pincell_fgmres}
\end{table}
\begin{figure}[!htbp]
    \centering
    \includegraphics[width=0.48\textwidth]{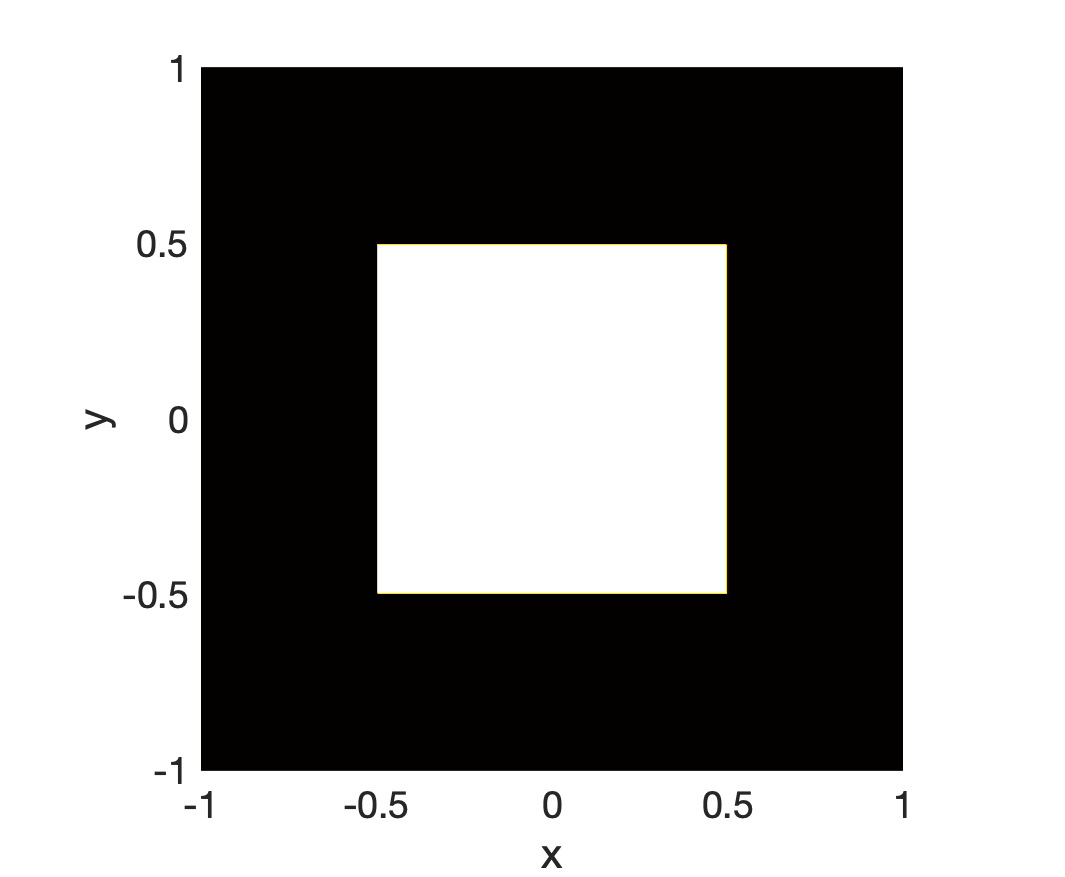}
    \includegraphics[width=0.48\textwidth]{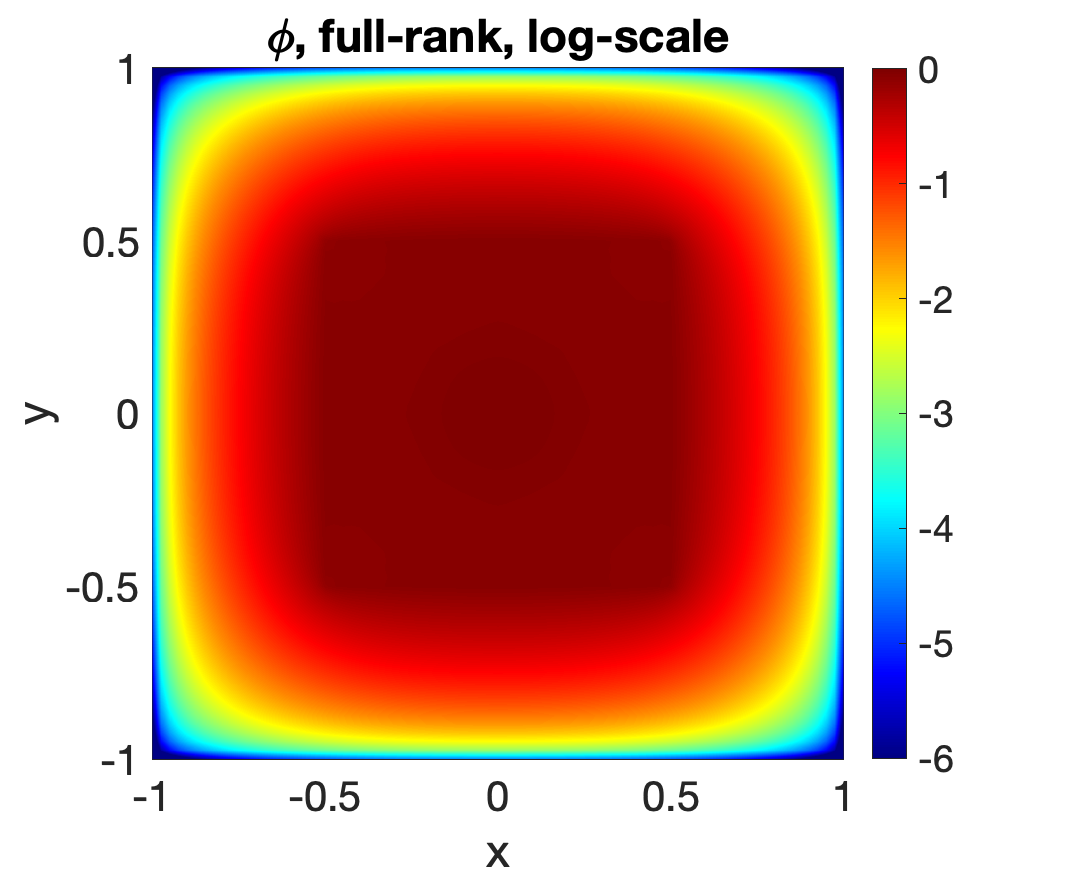}
    \includegraphics[width=0.45\textwidth]{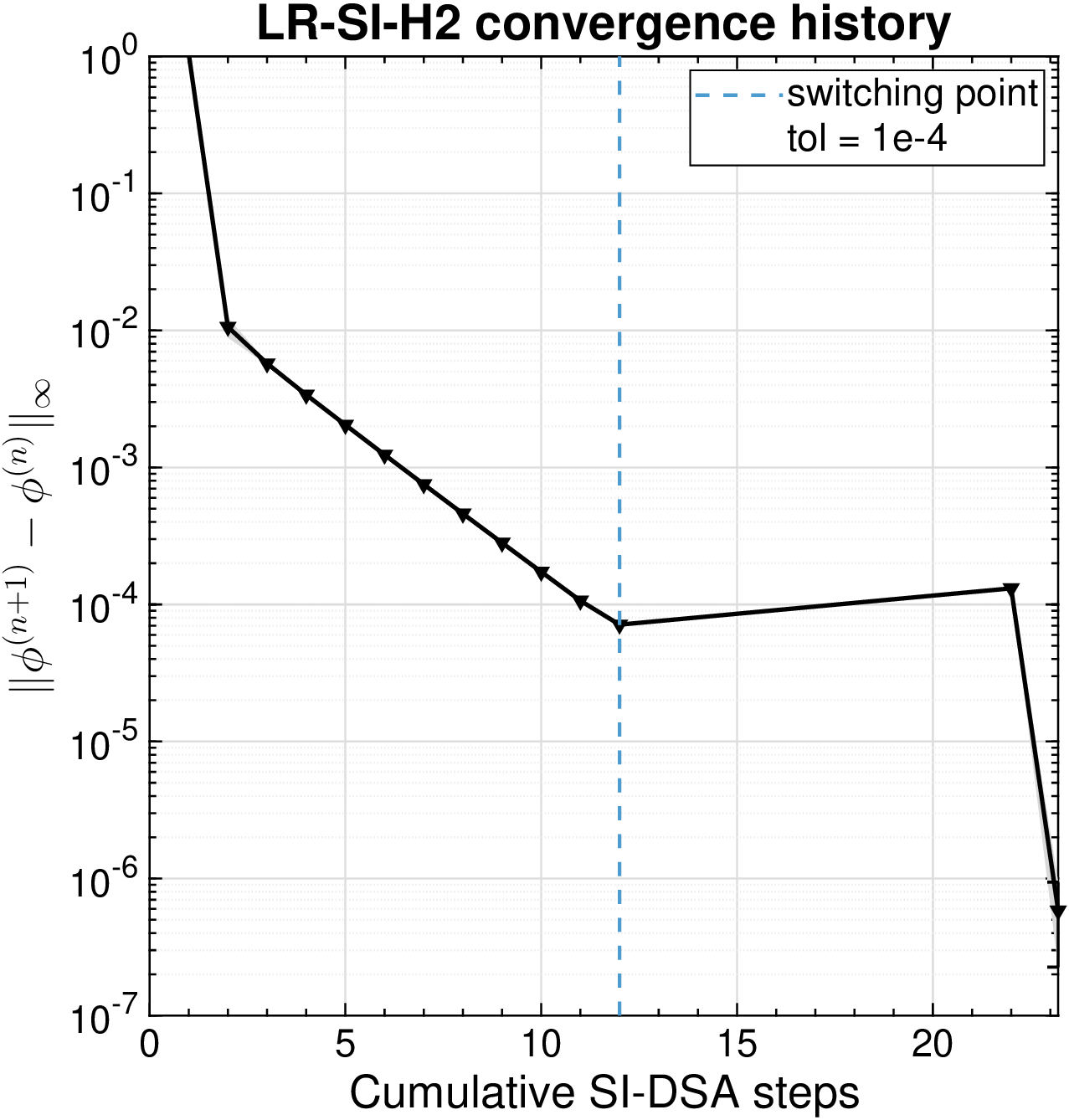}
    \includegraphics[width=0.45\textwidth]{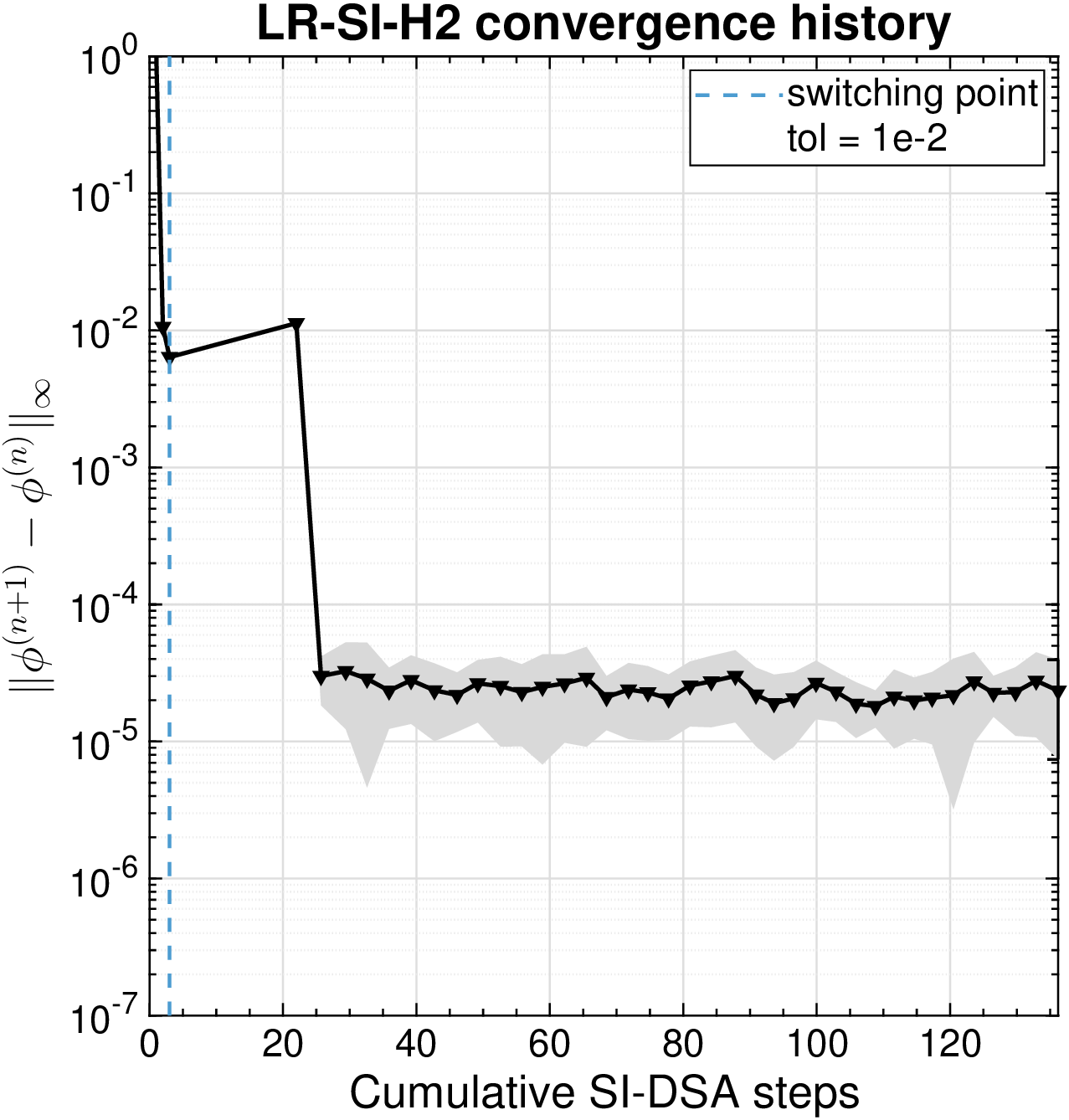}
      \caption{Pin-cell problem in Sec.~\ref{sec:pin-cell}. For all low-rank results, the plotted curves represent the mean over
ten random seeds, while the shaded regions indicate one standard
deviation about the mean. Top left: material configuration. Top right: reference scalar flux under log-scale. Bottom left and right: convergence history for the two-stage hybrid LR--SI--DSA with $\varepsilon_{\mathrm{switch}}=10^{-4}$ (left) and $\varepsilon_{\mathrm{switch}}=5\times10^{-3}$ (right).
}  
    \label{fig:pin-cell}
\end{figure}


\section{Conclusion}
\label{sec:conclusion}
We have developed two complementary approaches to improve sweep-based low-rank iterative solvers for steady-state RTEs. First, building on randomized-greedy low-rank SI--DSA \cite{Guo2026HighlyER}, we have introduced adaptive inner accuracy control to reduce intermediate ranks and sampling costs. A hybrid sampling strategy provides further acceleration by switching from randomized greedy sampling during early iterations to a steady-state extension of the DEIM-based DLR method \cite{haut2026efficient} near convergence.
Second, to address the reduced effectiveness of SI--DSA in multiscale problems with
strong material discontinuities, we developed an inexact FGMRES--DSA solver
for the integral scalar-flux reformulation of RTE. Sweep-based low-rank approximations via randomized greedy angular sampling are applied to efficiently construct the right-hand side and perform matrix--vector products.
By confining low-rank approximations in these operations, the
scalar-flux Krylov basis is not truncated and remains orthogonal in exact
arithmetic, enabling extension of inexact-GMRES theory \cite{simoncini2003theory} to guide the adaptive
accuracy control.

The numerical results demonstrate that adaptive accuracy control reduces angular sampling and intermediate ranks while maintaining the requested solution accuracy and largely preserving the full-rank convergence behavior, yielding greater speedups over full-rank solvers than fixed-accuracy low-rank approximations. For low-rank SI--DSA, these gains depend on choosing the maximum relaxation threshold to avoid additional outer iterations. The two-stage hybrid method provides modest further savings when the angular basis is sufficiently representative at the switch. Low-rank FGMRES--DSA is more robust to aggressive relaxation and consistently outperforms both its fixed-accuracy low-rank and full-rank counterparts. For the challenging pin-cell and lattice problems, adaptively relaxed low-rank FGMRES--DSA is also substantially more efficient than its SI--DSA counterpart.


\section*{Acknowledgments}
Z. Peng was partially supported by the Hong Kong Research Grants Council grants Early Career Scheme 26302724 and General Research Fund 16306825, 16307226. During the preparation of this work the author(s) used ChatGPT and Claude to check grammar errors,  improve readability and ChatGPT to assist code implementation. After using this tool, the authors reviewed and edited the content as needed and take full responsibility for the content of the published article.

\appendix

\section{Chebyshev-Gauss angular quadrature}\label{app:quadrature}
We use a tensor-product Chebyshev--Gauss quadrature on the sphere, given by
\begin{equation}
  \BOmega_j
  = \left(
      \cos\theta_{j_1}\sqrt{1-\mu_{j_2}^2},
      \sin\theta_{j_1}\sqrt{1-\mu_{j_2}^2},
      \mu_{j_2}
    \right),
  \qquad
  \omega_j=\omega_{\theta,j_1}\omega_{\mu,j_2}.
\end{equation}
\section{DEIM-based angle sampling\label{app:deim-sampling-alg}}
\begin{algorithm}[H]
\caption{DEIM-based angle sampling.\label{alg:deim-sampling}}
\begin{algorithmic}[1]
\STATE{\textbf{Input}: Angular basis
$\BV^{(n)}
=[\,\bv_1^{(n)},\ldots,\bv_r^{(n)}\,]
\in\mathbb{R}^{N_\Omega\times r}$.}
\STATE{\textbf{Initialization}: Select
$p_1:=\underset{1\leq p\leq N_\Omega}{\arg\max}\,
\left|(\bv_1^{(n)})_p\right|,
\qquad
\mathcal{I}:=\{p_1\}.$}
\FOR{$j=2,3,\dots,r$}
\STATE{Set
$\BV_{j-1}^{(n)}
:=[\,\bv_1^{(n)},\ldots,\bv_{j-1}^{(n)}\,]$
and
$\BPi_{j-1}:=[\,\be_{p_1},\ldots,\be_{p_{j-1}}\,]$.}
\STATE{\textbf{Interpolation}: Solve
$\big(\BPi_{j-1}^{T}\BV_{j-1}^{(n)}\big)\bc^{(j-1)}
=
\BPi_{j-1}^{T}\bv_j^{(n)}.
$}
\STATE{\textbf{Residual}: Compute
$\br_j
:=
\bv_j^{(n)}-\BV_{j-1}^{(n)}\bc^{(j-1)}.$}
\STATE{\textbf{Angle selection}: Select
$p_j:=
\underset{1\leq p\leq N_\Omega}{\arg\max}\,
|(\br_j)_p|,
\qquad
\mathcal{I}:=\mathcal{I}\cup\{p_j\}.
$}
\ENDFOR
\STATE{\textbf{Output}: Sampled-angle index set
$\mathcal{I}=\{p_1,\ldots,p_r\}$, corresponding to
$\BOmega_{p_1},\ldots,\BOmega_{p_r}$.}
\end{algorithmic}
\end{algorithm}

\section{Proof of Theorem~\ref{thm:inexact-richardson}}\label{app:proof-inexact-richardson}

\begin{proof}[Proof of Theorem~\ref{thm:inexact-richardson}]
Define the error by
$\be^{(n)}:=\bx^{(n)}-\bx^*.$
Since $\BA\bx^*=\bb$, the inexact iteration can be written as
\begin{align}
    \be^{(n+1)}
    &=
    \be^{(n)}
    +
    \alpha\BM^{-1}
    \left(
        \bb-\BA\bx^{(n)}+\widetilde{\br}^{(n)}-\br^{(n)}
    \right)
    \notag\\
    &=
    \left(
        \BI-\alpha\BM^{-1}\BA
    \right)
    \be^{(n)}
    +
    \alpha\BM^{-1}(\widetilde{\br}^{(n)}-\br^{(n)}).
\end{align}
Taking the $\BM$-norm gives
\begin{equation}
    \|
        \be^{(n+1)}
   \|_{\BM}
    \leq
    \rho
   \|
        \be^{(n)}
    \|_{\BM}
    +
    \alpha
    \|
        \widetilde{\br}^{(n)}-\br^{(n)}
    \|_{\BM^{-1}}.
    \label{eq:M-norm-relation}
\end{equation}
By \eqref{eq:relative-perturbation-condition},
\begin{align}
    \|
        \widetilde{\br}^{(n)}-\br^{(n)}
    \|_{\BM^{-1}}
    &\leq
    \eta^{(n)}
    \|
        \br^{(n)}
    \|_{\BM^{-1}}
    =
    \eta^{(n)}
    \|
        \BA\be^{(n)}
    \|_{\BM^{-1}}
    \notag\\
    &=
    \eta^{(n)}
    \|
        \BM^{-1}\BA\be^{(n)}
    \|_{\BM}
    \leq
    \eta^{(n)}K
    \|
        \be^{(n)}
   \|_{\BM}.
\end{align}
Substituting this estimate into \eqref{eq:M-norm-relation},
\[
    \|
        \be^{(n+1)}
    \|_{\BM}
    \leq
    \left(
        \rho+\alpha K\eta^{(n)}
    \right)
    \|
        \be^{(n)}
    \|_{\BM}.
\]
Hence,
$\eta^{(n)}
    \leq
    \overline{\eta}
    <
    \frac{1-\rho}{\alpha K}$ implies
$\rho+\alpha K\overline{\eta}<1,$
and the iteration remains contractive.
\end{proof}

\section{Proof of Theorem~\ref{thm:residual-gap} \label{app:gmres-proof}}
\begin{proof}[Proof of Theorem~\ref{thm:residual-gap}]
The argument adapts the residual-gap analysis in Proposition~4.1,
Lemma~5.1, and Theorem~5.3 of \cite{simoncini2003theory}
to flexible GMRES with a perturbed right-hand side
and an inexact initial matrix--vector product.
By assumption, 
\[
\bepsilon_0
:=
\BA_{\mathrm{lr}}(\bphi^{(0)};\tau_0)-\BA\bphi^{(0)}, \quad \|\bepsilon_0\|_2\leq\varepsilon_0,\quad
\widetilde{\bb}-\BA\bphi^{(0)}
=
\br^{(0)}+\bepsilon_0
=
\beta\bq^{(1)}+\bepsilon_0.
\]
Recall that
\[
\BA\BZ_m
=
\BQ_{m+1}\overline{\BH}_m-\BE_m,\;
\bphi^{(m)}=\bphi^{(0)}+\BZ_m\by^{(m)}.
\]
The true residual satisfies
\begin{align}
\br_{\mathrm{true}}^{(m)}
&=
\bb-\BA\bphi^{(m)}
=
\bb-\widetilde{\bb}
+
\widetilde{\bb}-\BA(\bphi^{(0)}+\BZ_m\by^{(m)})
\notag\\
&=\bb-\widetilde{\bb}+(\widetilde{\bb}-\BA\bphi^{(0)})-\BA\BZ_m\by^{(m)}
\notag\\
&=
\bb-\widetilde{\bb}
+
\bepsilon_0
+
\beta\bq^{(1)}
-
\left(
\BQ_{m+1}\overline{\BH}_m-\BE_m
\right)\by^{(m)}
\notag\\
&=
\bb-\widetilde{\bb}
+
\bepsilon_0
+
\BQ_{m+1}
\left(
\beta\be_1-\overline{\BH}_m\by^{(m)}
\right)
+
\BE_m\by^{(m)}
\notag\\
&=
\bb-\widetilde{\bb}
+
\bepsilon_0
+
\br^{(m)}
+
\BE_m\by^{(m)}.
\end{align}
Therefore,
\begin{equation}
\br_{\mathrm{true}}^{(m)}-\br^{(m)}
=
\bb-\widetilde{\bb}
+
\bepsilon_0
+
\BE_m\by^{(m)}.
\label{eq:app_residual_gap_identity}
\end{equation}

Though Lemma~5.1 of \cite{simoncini2003theory} is stated for
inexact GMRES, its proof depends only on the nested upper-Hessenberg
least-squares problems. The FGMRES least-squares problems have the
same structure; the flexible basis $\BZ_m$ does not enter the
coefficient estimate. Hence, the same argument gives the $k$-th element of $\by^{(m)}$ satisfying
\begin{equation}
\left|
\left(\by^{(m)}\right)_k
\right|
\leq
\frac{
\|\br^{(k-1)}\|_2
}{
\sigma_m(\overline{\BH}_m)
},
\qquad k=1,\ldots,m.
\label{eq:app_y_component_bound}
\end{equation}
Here, the orthonormality of $\BQ_k$ implies that
\[
\|\br^{(k-1)}\|_2
=
\left\|
\beta\be_1-\overline{\BH}_{k-1}\by^{(k-1)}
\right\|_2,
\qquad k=2,\ldots,m,
\]
while $\|\br^{(0)}\|_2=\beta$.
Thus, $\br^{(k-1)}$ is precisely the computed residual appearing
in the coefficient estimate.
It follows from \eqref{eq:app_y_component_bound} that
\begin{align}
\|\BE_m\by^{(m)}\|_2
&=
\|
\sum_{k=1}^m
\bepsilon_k\left(\by^{(m)}\right)_k
\|_2
\leq
\sum_{k=1}^m
\|\bepsilon_k\|_2
\left|
\left(\by^{(m)}\right)_k
\right|\notag\\
&\leq
\frac{1}{\sigma_m(\overline{\BH}_m)}
\sum_{k=1}^m
\|\bepsilon_k\|_2
\|\br^{(k-1)}\|_2.
\end{align}
Applying condition~\eqref{eq:gmres-error-constraint} to each term gives
$\|\BE_m\by^{(m)}\|_2
\leq
\varepsilon_{\mathrm{gap}}-2\varepsilon_0.$
Finally, using \eqref{eq:app_residual_gap_identity},
$\|\bb-\widetilde{\bb}\|_2\leq\varepsilon_0$,
and $\|\bepsilon_0\|_2\leq\varepsilon_0$, we obtain
\begin{align}
\|
\br_{\mathrm{true}}^{(m)}-\br^{(m)}
\|_2
&\leq
\|\bb-\widetilde{\bb}\|_2
+
\|\bepsilon_0\|_2
+
\|\BE_m\by^{(m)}\|_2
\notag\\
&\leq
2\varepsilon_0
+
\left(
\varepsilon_{\mathrm{gap}}-2\varepsilon_0
\right)
=
\varepsilon_{\mathrm{gap}}.
\end{align}
\end{proof}
\section{Notations\label{sec:notation}}
\begin{table}[H]
\centering
\footnotesize
\setlength{\tabcolsep}{4pt}
\begin{tabular}{@{}ll@{}}
\toprule
\multicolumn{2}{@{}l}{\emph{Discretization}}\\
$N_{\bx}$, $N_\Omega$ & numbers of spatial DOFs and quadrature directions\\
$\BOmega_j$, $\omega_j$ & $j$-th quadrature direction and its normalized weight\\
$\bw$, $\BD_\Omega$ & weight vector $(\omega_1,\ldots,\omega_{N_\Omega})^T$ and $\operatorname{diag}(\bw)$\\
$\bpsi_j$, $\bphi$, $\BPsi$ & discrete angular flux along $\BOmega_j$; discrete scalar flux; $\BPsi=[\,\bpsi_1,\ldots,\bpsi_{N_\Omega}]\in\mathbb{R}^{N_{\bx}\times N_\Omega}$\\
$\bG$, $\bg_j^{\mathrm{bc}}$ & fixed source; inflow-boundary contribution for $\BOmega_j$\\
\midrule
\multicolumn{2}{@{}l}{\emph{Operators}}\\
$\BD_j$, $\BSigma_t$, $\BSigma_s$ & discrete streaming, total and scattering operators\\
$\BA_{\mathrm{diff}}$ & discrete diffusion operator of the DSA correction\\
$\BM_{\mathrm{DSA}}$ & DSA preconditioner, $\BM_{\mathrm{DSA}}^{-1}=\BI+\BA_{\mathrm{diff}}^{-1}\BSigma_s$\\
$\BT$, $\bb$ & $\sum_j\omega_j(\BD_j+\BSigma_t)^{-1}$ and the right-hand side of integral formulation \eqref{eq:memory-efficient-reformulation}\\
$\widetilde{\BT}_{\mathrm{lr}}$, $\widetilde{\bb}$ & their sweep-based low-rank approximations\\
$\BA_{\mathrm{lr}}(\,\cdot\,;\tau)$ & inexact matrix--vector product \eqref{eq:lr-matvec-operator}\\
\midrule
\multicolumn{2}{@{}l}{\emph{Low-rank factors}}\\
$r$ & rank of the low-rank factorization\\
$\BX$, $\BS$, $\BV$, $\BC$ & spatial basis, core, angular basis, angular coefficients $\BPsi\approx\BX\BC=\BX\BS\BV^{T}$\\
$\BP^{(n,i)}$ & rank-$p$ increment of the spatial basis in Alg.~\ref{alg:lr-si}\\
$\BB_j$, $\BM_s$, $\bff_j$ & projected operators and source of \eqref{eq:deim_row_equation}\\
$\BL_j$, $\overline{\BL}$ & reduced angular coefficients and their quadrature average\\
\midrule
\multicolumn{2}{@{}l}{\emph{Krylov quantities} (Sec.~\ref{sec:ta-lr-fgmres})}\\
$\BQ_{m+1}$, $\bq^{(k)}$ & orthonormal Arnoldi basis and its $k$-th column\\
$\BZ_k$, $\bz^{(k)}$ & flexible (preconditioned) basis and its $k$-th column\\
$\overline{\BH}_k$, $\by^{(k)}$ & Hessenberg matrix and least-squares coefficients\\
$\BE_m$, $\bepsilon_k$ & matrix of matrix--vector product errors and its $k$-th column\\
\bottomrule
\end{tabular}
\caption{Notation used throughout the paper. Latin vectors and matrices are set
in upright bold, Greek ones in bold italic.}
\label{tab:notation}
\end{table}

\bibliographystyle{siam}
\bibliography{ref}

@article{Adams2002FastIM,
  title={{Fast iterative methods for discrete-ordinates particle transport calculations}},
  author={Marvin L. Adams and Edward W. Larsen},
  journal={Progress in Nuclear Energy},
  year={2002},
  volume={40},
  pages={3-159}
}

@article{larsen2009advances,
  title={{Advances in discrete-ordinates methodology}},
  author={Larsen, Edward W and Morel, Jim E},
  journal={Nuclear computational science: A century in review},
  pages={1--84},
  year={2009},
  publisher={Springer}
}

@article{guermond2010asymptotic,
  title={{Asymptotic analysis of upwind discontinuous Galerkin approximation of the radiative transport equation in the diffusive limit}},
  author={Guermond, Jean-Luc and Kanschat, Guido},
  journal={SIAM Journal on Numerical Analysis},
  volume={48},
  number={1},
  pages={53--78},
  year={2010},
  publisher={SIAM}
}

@article{adams2001discontinuous,
  title={{Discontinuous finite element transport solutions in thick diffusive problems}},
  author={Adams, Marvin L},
  journal={Nuclear science and engineering},
  volume={137},
  number={3},
  pages={298--333},
  year={2001},
  publisher={Taylor \& Francis}
}

@article{chen2009integrated,
  title={{An integrated finite element method package in MATLAB}},
  author={Chen, L},
  journal={California: University of California at Irvine},
  volume={10},
  year={2009}
}

@article{guo2022low,
  title={{A low rank tensor representation of linear transport and nonlinear Vlasov solutions and their associated flow maps}},
  author={Guo, Wei and Qiu, Jing-Mei},
  journal={Journal of Computational Physics},
  volume={458},
  pages={111089},
  year={2022},
  publisher={Elsevier}
}

@book{pomraning2005equations,
  title={{The equations of radiation hydrodynamics}},
  author={Pomraning, Gerald C},
  year={2005},
  publisher={Courier Corporation}
}

@article{birken2015termination,
  title={{Termination criteria for inexact fixed-point schemes}},
  author={Birken, Philipp},
  journal={Numerical Linear Algebra with Applications},
  volume={22},
  number={4},
  pages={702--716},
  year={2015},
  publisher={Wiley Online Library}
}

@article{bachmayr2023low,
  title={{Low-rank tensor methods for partial differential equations}},
  author={Bachmayr, Markus},
  journal={Acta Numerica},
  volume={32},
  pages={1--121},
  year={2023},
  publisher={Cambridge University Press}
}

@article{kressner2011low,
  title={{Low-rank tensor Krylov subspace methods for parametrized linear systems}},
  author={Kressner, Daniel and Tobler, Christine},
  journal={SIAM Journal on Matrix Analysis and Applications},
  volume={32},
  number={4},
  pages={1288--1316},
  year={2011},
  publisher={SIAM}
}

@article{peng2023sweep,
  title={{A sweep-based low-rank method for the discrete ordinate transport equation}},
  author={Peng, Zhuogang and McClarren, Ryan G},
  journal={Journal of Computational Physics},
  volume={473},
  pages={111748},
  year={2023},
  publisher={Elsevier}
}

@article{peng2021high,
  title={{A high-order/low-order (HOLO) algorithm for preserving conservation in time-dependent low-rank transport calculations}},
  author={Peng, Zhuogang and McClarren, Ryan G},
  journal={Journal of Computational Physics},
  volume={447},
  pages={110672},
  year={2021},
  publisher={Elsevier}
}

@article{peng2020low,
  title={{A low-rank method for two-dimensional time-dependent radiation transport calculations}},
  author={Peng, Zhuogang and McClarren, Ryan G and Frank, Martin},
  journal={Journal of Computational Physics},
  volume={421},
  pages={109735},
  year={2020},
  publisher={Elsevier}
}

@article{kusch2021robust,
  title={{A robust collision source method for rank adaptive dynamical low-rank approximation in radiation therapy}},
  author={Kusch, Jonas and Stammer, Pia},
  journal={arXiv preprint arXiv:2111.07160},
  year={2021}
}

@article{einkemmer2024asymptotic,
  title={{Asymptotic-preserving and energy stable dynamical low-rank approximation}},
  author={Einkemmer, Lukas and Hu, Jingwei and Kusch, Jonas},
  journal={SIAM Journal on Numerical Analysis},
  volume={62},
  number={1},
  pages={73--92},
  year={2024},
  publisher={SIAM}
}

@article{ding2021dynamical,
  title={{Dynamical low-rank integrator for the linear Boltzmann equation: error analysis in the diffusion limit}},
  author={Ding, Zhiyan and Einkemmer, Lukas and Li, Qin},
  journal={SIAM Journal on Numerical Analysis},
  volume={59},
  number={4},
  pages={2254--2285},
  year={2021},
  publisher={SIAM}
}

@article{kusch2023robust,
  title={{A robust collision source method for rank adaptive dynamical low-rank approximation in radiation therapy}},
  author={Kusch, Jonas and Stammer, Pia},
  journal={ESAIM: Mathematical Modelling and Numerical Analysis},
  volume={57},
  number={2},
  pages={865--891},
  year={2023},
  publisher={EDP Sciences}
}

@article{sands2024high,
  title={{High-order adaptive rank integrators for multi-scale linear kinetic transport equations in the hierarchical Tucker format}},
  author={Sands, William A and Guo, Wei and Qiu, Jing-Mei and Xiong, Tao},
  journal={arXiv preprint arXiv:2406.19479},
  year={2024}
}

@article{hauck2023predictor,
  title={{A predictor-corrector strategy for adaptivity in dynamical low-rank approximations}},
  author={Hauck, Cory D and Schnake, Stefan},
  journal={SIAM Journal on Matrix Analysis and Applications},
  volume={44},
  number={3},
  pages={971--1005},
  year={2023},
  publisher={SIAM}
}

@article{bachmayr2017iterative,
  title={{Iterative methods based on soft thresholding of hierarchical tensors}},
  author={Bachmayr, Markus and Schneider, Reinhold},
  journal={Foundations of Computational Mathematics},
  volume={17},
  number={4},
  pages={1037--1083},
  year={2017},
  publisher={Springer}
}

@article{einkemmer2025review,
  title={{A review of low-rank methods for time-dependent kinetic simulations}},
  author={Einkemmer, Lukas and Kormann, Katharina and Kusch, Jonas and McClarren, Ryan G and Qiu, Jing-Mei},
  journal={Journal of Computational Physics},
  pages={114191},
  year={2025},
  publisher={Elsevier}
}

@article{adams1992diffusion,
  title={Diffusion synthetic acceleration of discontinuous finite element transport iterations},
  author={Adams, Marvin L and Martin, William R},
  journal={Nuclear science and engineering},
  volume={111},
  number={2},
  pages={145--167},
  year={1992},
  publisher={Taylor \& Francis}
}

@article{warsa2004krylov,
  title={{Krylov iterative methods and the degraded effectiveness of diffusion synthetic acceleration for multidimensional SN calculations in problems with material discontinuities}},
  author={Warsa, James S and Wareing, Todd A and Morel, Jim E},
  journal={Nuclear science and engineering},
  volume={147},
  number={3},
  pages={218--248},
  year={2004},
  publisher={Taylor \& Francis}
}

@article{dolgov2013tt,
  title={{TT-GMRES: solution to a linear system in the structured tensor format}},
  author={Dolgov, Sergey V},
  journal={Russian Journal of Numerical Analysis and Mathematical Modelling},
  volume={28},
  number={2},
  pages={149--172},
  year={2013},
  publisher={De Gruyter}
}

@article{dolgov2014alternating,
  title={{Alternating minimal energy methods for linear systems in higher dimensions}},
  author={Dolgov, Sergey V and Savostyanov, Dmitry V},
  journal={SIAM Journal on Scientific Computing},
  volume={36},
  number={5},
  pages={A2248--A2271},
  year={2014},
  publisher={SIAM}
}

@article{appelo2025lraa,
  title={{lrAA: Low-Rank Anderson Acceleration}},
  author={Appelo, Daniel and Cheng, Yingda},
  journal={arXiv preprint arXiv:2503.03909},
  year={2025}
}

@article{haut2026efficient,
  title={{Efficient SN-like and PN-like Dynamic Low Rank methods for Thermal Radiative Transfer}},
  author={Haut, Terry and Loffeld, John and Einkemmer, Lukas and Guthrey, Pierson and Brunner, Stefan and Schill, William},
  journal={arXiv preprint arXiv:2601.18705},
  year={2026}
}

@article{ceruti2022rank,
  title={{A rank-adaptive robust integrator for dynamical low-rank approximation}},
  author={Ceruti, Gianluca and Kusch, Jonas and Lubich, Christian},
  journal={BIT Numerical Mathematics},
  volume={62},
  number={4},
  pages={1149--1174},
  year={2022},
  publisher={Springer}
}

@article{koch2007dynamical,
  title={{Dynamical low-rank approximation}},
  author={Koch, Othmar and Lubich, Christian},
  journal={SIAM Journal on Matrix Analysis and Applications},
  volume={29},
  number={2},
  pages={434--454},
  year={2007},
  publisher={SIAM}
}

@article{dektor2025interpolatory,
  title={{Interpolatory dynamical low-rank approximation for the 3+ 3d Boltzmann--BGK equation}},
  author={Dektor, Alec and Einkemmer, Lukas},
  journal={Journal of Computational Physics},
  pages={114515},
  year={2025},
  publisher={Elsevier}
}

@article{dektor2021rank,
  title={{Rank-adaptive tensor methods for high-dimensional nonlinear PDEs}},
  author={Dektor, Alec and Rodgers, Abram and Venturi, Daniele},
  journal={Journal of Scientific Computing},
  volume={88},
  number={2},
  pages={36},
  year={2021},
  publisher={Springer}
}

@article{brunner2026domain,
title = {{Domain decomposition dynamical low-rank for multi-dimensional radiative transfer equations}},
journal = {Journal of Computational Physics},
pages = {115397},
year = {2026},
issn = {0021-9991},
doi = {https://doi.org/10.1016/j.jcp.2026.115397},
url = {https://www.sciencedirect.com/science/article/pii/S0021999126007461},
author = {Stefan Brunner and Lukas Einkemmer and Terry Haut},
}

@article{daniel1976reorthogonalization,
  title={{Reorthogonalization and stable algorithms for updating the Gram-Schmidt QR factorization}},
  author={Daniel, James W and Gragg, Walter Bill and Kaufman, Linda and Stewart, Gilbert W},
  journal={Mathematics of Computation},
  volume={30},
  number={136},
  pages={772--795},
  year={1976}
}

@article{deshpande2026multigroup,
  title={{Multigroup Thermal Radiation Transport with Tensor Trains}},
  author={Deshpande, Aditya S and Mullen, Patrick D and Gorodetsky, Alex A and Dolence, Joshua C and Meyer, Chad D and Miller, Jonah M and Roberts, Luke F},
  journal={The Astrophysical Journal},
  volume={1003},
  number={2},
  pages={193},
  year={2026},
  publisher={The American Astronomical Society}
}

@article{truong2024tensor,
  title={Tensor networks for solving the time-independent Boltzmann neutron transport equation},
  author={Truong, Duc P and Ortega, Mario I and Boureima, Ismael and Manzini, Gianmarco and Rasmussen, Kim {\O} and Alexandrov, Boian S},
  journal={Journal of computational physics},
  volume={507},
  pages={112943},
  year={2024},
  publisher={Elsevier}
}

@article{li2026asymptotic,
  title={An Asymptotic-Preserving Dynamical Low-Rank Semi-Lagrangian Method for Multiscale Linear Kinetic Transport Equations},
  author={Li, Shun and Jiang, Yan and Zhang, Mengping and Xiong, Tao},
  journal={ arXiv:2607.27736},
  year={2026}
}

@article{Guo2026HighlyER,
  title={{Highly Efficient Rank-Adaptive Sweep-based SI-DSA for the Radiative Transfer Equation via Mild Space Augmentation}},
  author={Wei Guo and Zhichao Peng},
  journal={ArXiv},
  year={2026},
  volume={abs/2603.25233},
  url={https://api.semanticscholar.org/CorpusID:286789887}
}

@article{simoncini2003theory,
author = {Simoncini, Valeria and Szyld, Daniel B.},
title = {Theory of Inexact Krylov Subspace Methods and Applications to Scientific Computing},
journal = {SIAM Journal on Scientific Computing},
volume = {25},
number = {2},
pages = {454-477},
year = {2003},
doi = {10.1137/S1064827502406415},

URL = { 
    
        https://doi.org/10.1137/S1064827502406415
    
    

},
eprint = { 
    
        https://doi.org/10.1137/S1064827502406415
    
    

}
}

@article{palitta2021convergence,
  author  = {Palitta, Davide and K{\"u}rschner, Patrick},
  title   = {On the convergence of {K}rylov methods with low-rank truncations},
  journal = {Numerical Algorithms},
  volume  = {88}, number = {3}, pages = {1383--1417}, year = {2021},
  doi     = {10.1007/s11075-021-01080-2},
}

@article{kormann2015semi,
  title={{A semi-Lagrangian Vlasov solver in tensor train format}},
  author={Kormann, Katharina},
  journal={SIAM Journal on Scientific Computing},
  volume={37},
  number={4},
  pages={B613--B632},
  year={2015},
  publisher={SIAM}
}

@article{bachmayr2026low,
  title={{Low-rank tensor product Richardson iteration for radiative transfer in plane-parallel geometry}},
  author={Bachmayr, Markus and Bardin, Riccardo and Schlottbom, Matthias},
  journal={SIAM Journal on Numerical Analysis},
  volume={64},
  number={1},
  pages={277--302},
  year={2026},
  publisher={SIAM}
}

@article{sands2025adaptive,
  title={{An adaptive-rank approach with greedy sampling for multi-scale BGK equations}},
  author={Sands, William A and Qiu, Jing-Mei and Hayes, Daniel and Zheng, Nanyi},
  journal={Journal of Computational Physics},
  pages={114523},
  year={2025},
  publisher={Elsevier}
}

@article{ghahremani2024cross,
  title={{Cross interpolation for solving high-dimensional dynamical systems on low-rank Tucker and tensor train manifolds}},
  author={Ghahremani, Behzad and Babaee, Hessam},
  journal={Computer Methods in Applied Mechanics and Engineering},
  volume={432},
  pages={117385},
  year={2024},
  publisher={Elsevier}
}

@article{dektor2025collocation,
  title={{Collocation methods for nonlinear differential equations on low-rank manifolds}},
  author={Dektor, Alec},
  journal={Linear Algebra and its Applications},
  volume={705},
  pages={143--184},
  year={2025},
  publisher={Elsevier}
}

@article{ortega2026tensor,
  title={{Tensor train neutron transport applied to nuclear reactor problems}},
  author={Ortega, Mario I and Truong, Duc P and Boureima, Ismael and Rasmussen, Kim {\O} and Alexandrov, Boian S},
  journal={Nucl. Sci. Eng.},
  volume={200},
  number={4},
  pages={800--820},
  year={2026},
  publisher={Taylor \& Francis}
}

@article{carrel2025interpolatory,
  title={{Interpolatory dynamical low-rank approximation: theoretical foundations and algorithms}},
  author={Carrel, Benjamin and Kressner, Daniel and Lam, Hei Yin and Vandereycken, Bart},
  journal={arXiv:2510.19518},
  year={2025}
}

@article{guo2026inexact,
  title={{An Inexact Low-Rank Source Iteration for Steady-State Radiative Transfer Equation with Diffusion Synthetic Acceleration}},
  author={Guo, Wei and Peng, Zhichao},
  journal={Journal of Computational Physics},
  pages={115427},
  year={2026},
  publisher={Elsevier}
}

\end{document}